\documentclass[a4paper,10pt,notitlepage,reqno]{article}
\usepackage[english]{babel}
\usepackage{amssymb,amsthm,bm} 
\usepackage{mathtools}
\usepackage{mathrsfs}
\usepackage{enumitem}
\usepackage{authblk}
\usepackage{cite}
\usepackage{xcolor}
\usepackage{tikz}
\usepackage{stackrel}
\usepackage{geometry}
\usepackage{hyperref}
\theoremstyle{plain} 
\newtheorem{theorem}{Theorem}[section]
\newtheorem*{theorem*}{Theorem}
\newtheorem{lemma}{Lemma}[section]
\newtheorem{proposition}{Proposition}[section]
\newtheorem{corollary}{Corollary}[section]
\newtheorem*{corollary*}{Corollary}

\theoremstyle{definition}
\newtheorem*{definition}{Definition}

\theoremstyle{remark}
\newtheorem{remark}{Remark}[section]

\newtheorem{es}{\bf Example}
\numberwithin{equation}{section}

\usepackage{color}
\usepackage[normalem]{ulem}
\definecolor{DarkGreen}{rgb}{0,0.5,0.1} 

\newcommand\soutD{\bgroup\markoverwith
{\textcolor{DarkGreen}{\rule[.5ex]{2pt}{1pt}}}\ULon}

\newcommand{\Hm}[1]{\leavevmode{\marginpar{\tiny%
$\hbox to 0mm{\hspace*{-1.5mm}$\leftarrow$\hss}%
\vcenter{\vrule depth 0.1mm height 0.1mm width \the\marginparwidth}%
\hbox to
0mm{\hss$\rightarrow$\hspace*{-1.5mm}}$\\\relax\raggedright #1}}}

\definecolor{Darkgblue}{rgb}{0.3,0.3,0.5}

\newcommand{\al}{\alpha}

\newcommand{\la}{\lambda}

\def\erre{\mathbb{R}}

\def\zeta{\mathbb{Z}}

\newcommand\rN{\erre^{N}}
\newcommand{\hn}{\mathbb{H}_n}

    \newcommand{\pa}{\partial}

\def\L{\mathcal{L}}

\newcommand{\tl}{~}
\DeclarePairedDelimiter{\abs}{\lvert}{\rvert}

\DeclareMathOperator{\Div}{div}
\newcommand{\norma}[1]{\left\lVert#1\right\rVert}

\newcommand{\diff}{\!\setminus\!}
\newcommand{\dyle}{\displaystyle}

\title{Weighted Poincaré inequality and Hardy improvements related to some degenerate elliptic differential operators}

\author[1]{Lorenzo D'Arca}

\affil[1]{Department of Mathematics “Guido Castelnuovo”, Sapienza University of Rome,\newline Piazzale Aldo Moro 5, Roma
00185, Italy; lorenzo.darca@uniroma1.it}

\begin{document}
\date{\small 15 July 2024}
\maketitle

\begin{abstract}
\noindent
In this paper, we characterize the sharp constant and maximizing functions for weighted Poincaré inequalities.
These results lead to refinements of Hardy's inequality obtained by adding remainder terms involving \(L^p\) norms.
We use techniques that avoid symmetric rearrangement argument, simplifying the analysis of these inequalities in both Euclidean and non-Euclidean contexts. Specifically, this method applies to a variety of settings, such as the Heisenberg group, various Carnot groups and operators expressed as sums of squares of vector fields. Significant examples include the Heisenberg-Greiner operator and the Baouendi-Grushin operator.
\end{abstract}

\section{Introduction}
In the field of functional inequalities, Poincaré inequality and Hardy inequality play a crucial role due to their importance in various areas of mathematics (see for example \cite{BE}, \cite{OK} and the references therein). 
In this paper, we characterize the sharp constant and the maximizing functions for certain weighted Poincaré inequalities.
These results are used to derive \(L^p\) generalizations of the Brezis-Vázquez improvement of Hardy's inequality, \cite{BV},
\begin{equation}\label{INTRO: EQ: Brezis-Vazquez}
\int_{\Omega} |\nabla u|^2 \, dx \geq \left(\frac{N-2}{2}\right)^2 \int_{\Omega} \frac{u^2}{|x|^2} \, dx + z_0^2 \left(\frac{\omega_N}{|\Omega|}\right)^{\frac{2}{N}} \int_{\Omega} |u|^2 \, dx 
\end{equation}

Here \( N \geq 2 \), \( z_0 \) denotes the first zero of the Bessel function \( J_0 \), while \(\omega_N\) and \( | \Omega | \) represent the \( N \)-dimensional Lebesgue measure of the unit sphere \( B \subset \mathbb{R}^N \) and the domain \( \Omega \), respectively.
In the Euclidean context, Gazzola, Grunau and Mitidieri extended in \cite{GGM} this inequality to the \( L^p \) setting. In \cite{BV} and \cite{GGM}, inequality (\ref{INTRO: EQ: Brezis-Vazquez}) and its \( L^p \) version are proven using a symmetric rearrangement argument to reduce the problem to a one-dimensional case. However, this technique does not apply to weighted inequalities and cannot be extended beyond the Euclidean setting. See Ghoussoub and Moradifam \cite{GM0} and \cite{GM} for another approach to proving \eqref{INTRO: EQ: Brezis-Vazquez}, where many other \( L^2 \) Hardy inequalities with general weights \( V \) and \( W \) are also established.
Refined versions of Hardy inequalities seem to have first appeared in \cite{Ma}. Remainder terms also appear in other Sobolev inequalities \cite{BL}, \cite{GG}.
In this paper, we aim to study weighted versions of inequality (\ref{INTRO: EQ: Brezis-Vazquez}) and, in doing so, we also examine weighted versions of Poincaré inequality. The work is conducted in a general setting that simultaneously captures the structure of the Euclidean context, the Heisenberg group, Carnot groups, and certain second-order differential operators such as the Baouendi-Grushin operator and the Heisenberg-Greiner operator.

We use techniques recently introduced by the author in \cite{LD}, which provide a simple and compact way to study functional inequalities like Hardy's and Rellich's in both Euclidean and non-Euclidean contexts. These techniques allow us to avoid using symmetric rearrangement argument, enabling us to handle the \( L^p \) and weighted versions in the aforementioned contexts.
We present the results within our general framework; however, to the best of the author's knowledge, they are novel even in the Euclidean context.

In Section \ref{SEC: Main Results}, we introduce the general setting and our main results. In Section \ref{SEC POINCARE}, we address weighted Poincaré inequalities:
\begin{equation}\label{INTRO: EQ: POINCARE}
\left(\frac{\nu_1(p,\theta)}{R}\right)^p \int_{B_R} \frac{|u|^p}{|x|^{N-\theta}} \, dx \leq \int_{B_R} \frac{|\nabla u|^p}{|x|^{N-\theta}} \, dx 
\end{equation}
Here \(\nu_1(p,\theta)\) is the first zero of an appropriate special function. For example, if \( p = \theta = 2 \), then \(\nu_1(2,2) = z_0\), the first zero of the Bessel function \( J_0 \).

In Section \ref{SEC Hardy IMPROVEMENT}, we propose two different ways to generalize inequality (\ref{INTRO: EQ: Brezis-Vazquez}):
\begin{align}
\label{INTRO: EQ: HARDY TIPE I}
&\int_{B_R} \frac{|\nabla u|^p}{|x|^{p(\theta - 1)}} \, dx \geq \left|\frac{N - p\theta}{p}\right|^p \int_{B_R} \frac{|u|^p}{|x|^{p\theta}} \, dx + \lambda_p \frac{\nu_1(p,p)^p}{R^p} \int_{B_R} \frac{|u|^p}{|x|^{p(\theta - 1)}} \, dx \\[10pt]
%
\label{INTRO: EQ: HARDY TIPE II}
&\int_{B_R} \frac{|\nabla u|^p}{|x|^{p(\theta - 1)}} \, dx \geq \left|\frac{N - p\theta}{p}\right|^p \int_{B_R} \frac{|u|^p}{|x|^{p\theta}} \, dx + \frac{2}{p} \left| \frac{N - p\theta}{p} \right|^{p-2} \frac{z_0^2}{R^2} \int_{B_R} \frac{|u|^p}{|x|^{p\theta - 2}} \, dx 
\end{align}

\(\lambda_p\) is a fixed constant that depends only on \(p\); in particular, for \(p = 2\), \(\lambda_p = 1\).
We note that for \( p = 2 \), inequalities (\ref{INTRO: EQ: HARDY TIPE I}) and (\ref{INTRO: EQ: HARDY TIPE II}) coincide with the weighted version of inequality (\ref{INTRO: EQ: Brezis-Vazquez}).



In Section \ref{SEC: Applicazioni particolari}, we discuss specific examples within our general framework. These include adaptations to the Euclidean context, the Heisenberg group, the Baouendi-Grushin operator, the Heisenberg-Greiner operator, and homogeneous Carnot groups.

Finally, in Appendix \ref{Appendice A}, we study in detail some nonlinear ODEs that are fundamental for inequality (\ref{INTRO: EQ: POINCARE}). These results should be well-known, but the author has not found a comprehensive reference in the literature. The arguments used are an adaptation to our case of ideas found in \cite{MCRW}, \cite{H},\cite{W}, \cite{O}. Appendix \ref{Appendice B} contains the proof of some simple but essential results for studying improvements (\ref{INTRO: EQ: HARDY TIPE I}) and (\ref{INTRO: EQ: HARDY TIPE II}).

\subsection{The main results}\label{SEC: Main Results}
In order to state our main results, we need to introduce some preliminary notations.\\\\
In \(\mathbb{R}^N\), we consider \(h \leq N\) vector fields \[ 
X_i = \sum_{j=1}^N \sigma_{i,j} \frac{\partial}{\partial x_j}, \quad i=1, \ldots, h, 
\]

and assume that \(\sigma_{i,j} \in C(\mathbb{R}^N)\) and \(\frac{\partial}{\partial x_j} \sigma_{i,j} \in C(\mathbb{R}^N)\). Under these assumptions, the formal adjoint of \(X_i\) is well defined as: 
\[ 
X_i^* = -\sum_{j=1}^N \frac{\partial}{\partial x_j} (\sigma_{i,j}\, \cdot\, ).
\] 
We denote by \(\nabla_\L\) the vector field 
\[ 
\nabla_\L \coloneqq (X_1, \ldots, X_h) = \sigma \nabla,
\] 
where \(\sigma\) is the \(h \times N\) matrix given by \((\sigma(x))_{i,j} = \sigma_{i,j}(x)\) and \(\nabla\) is the usual Euclidean gradient. \\ 
We also assume that there exists a family of dilations on \(\mathbb{R}^N\) 
\[ 
\delta_\lambda(x) = (\lambda^{\beta_1} x_1, \ldots, \lambda^{\beta_N} x_N) \quad \text{with} \quad \beta_j > 0, \]
with respect to which \(\nabla_\L\) is a vector field homogeneous of degree one. This means that for each \(i=1,\ldots,h\), for every smooth function \(f\), \(x \in \mathbb{R}^N\), and for every \(\lambda > 0\), we have \(X_i(f(\delta_\lambda(x))) = \lambda\, (X_i f)(\delta_\lambda(x))\).
Finally, we set \(Q = \beta_1 + \ldots + \beta_N\) as the homogeneous dimension defined by the dilations.\\
In the following we will use the notation 
\[ \Div_\L(\cdot) = -\nabla^*_\L \cdot = \Div(\sigma^T \cdot). 
\]
We can now define the counterpart of the Laplace and p-Laplace operators. 
We define the second-order differential operator 
\[ 
\L \coloneqq 
\Div(\sigma^T \sigma \nabla) = -\sum_{i=1}^h X_i^* X_i = -\nabla_\L^* \cdot \nabla_\L=\Div_\L(\nabla_\L), 
\]
and, for \(p>1\), we denote by \(\L_p\) the operator  
\[ 
\L_p(u) = 
-\nabla_\L^* \cdot (|\nabla_\L u|^{p-2}\nabla_\L u)=\Div_\L(|\nabla_\L u|^{p-2}\nabla_\L u). 
\]

\begin{es}[Euclidean Laplacian]
Let \(X_i = \partial_{x_i}\) for \(i=1, \ldots, N\). It is clear that \(\nabla_\L = (X_1, \ldots, X_N)\) is the usual Euclidean gradient and the corresponding operator \(\L_p\) is the usual p-Laplacian.\\
\end{es}

\begin{es}[Heisenberg sub-Laplacian]
Let \(N=2n+1\) and \((z,t)=(x,y,t)\in \mathbb{R}^{2n}\times\mathbb{R}\). Consider the vector fields
\[
X_i=\partial_{x_i}+2y_i\partial_t \quad \text{and} \quad Y_i=\partial_{y_i}-2x_i\partial_t \quad \text{for} \quad i=1,\ldots,n.
\]

Then, the \(2n \times (2n+1)\) matrix \(\sigma\) is given by
\[
\begin{pmatrix} 
\mathbf{I}_n & 0 & 2y \\
0 & \mathbf{I}_n & -2x 
\end{pmatrix}
\]

The corresponding vector field \(\nabla_{\hn}\) is the Heisenberg horizontal gradient , and the operator \(\L\) is the Heisenberg sub-Laplacian \( \L = \sum_{i=1}^n X_i^2 + Y_i^2 = \Delta_{\hn} \) 
(Here we are using that \(X_i^*=-X_i, Y_i^*=-Y_i\)).\\
\end{es}

Refer to section \ref{SEC: Applicazioni particolari} for a detailed discussion of these and other examples, such as Baouendi-Grushin type operator, Heisenberg-Greiner operator or sub-Laplacian on Carnot groups.\\

Let us now introduce some preliminary results on certain second-order nonlinear ODEs. For the reader's convenience, detailed proofs of Proposition \ref{1: PROP: Esistenza Unicità e Oscillazioni} and Proposition \ref{1: THM: Teorema Fondamentale} can be found in Appendix \ref{Appendice A}.

In what follows, we use the notation \(\phi_p(t) = |t|^{p-2}t\).

\begin{proposition}\label{1: PROP: Esistenza Unicità e Oscillazioni}
Let \(\theta \geq 1\). The problem
\[
\text{\bf{(P)}}\quad
\begin{cases} 
(r^{\theta-1}\phi_p(\varphi'))'+r^{\theta-1}\phi_p(\varphi)=0, & r \geq 0 \\
\varphi(0)=1, \quad \varphi'(0)=0
\end{cases}
\]
admits a unique solution \( \varphi \in C^1[0,+\infty) \) with \( r^{\theta-1}\phi_p(\varphi') \in C^1[0,+\infty) \).
Furthermore, this solution is oscillatory, meaning for every \( r > 0 \) there exists a \( t > r \) such that \( \varphi(t) = 0 \), and each zero is simple (if \( \varphi(z_0) = 0 \), then \( \varphi'(z_0) \neq 0 \)).
\end{proposition}

\begin{remark}\label{1: OSS: Osservazione Fondamentale Zeri di phi}
From the previous proposition, it follows that the zeros of the function \( \varphi \) are a countable quantity, and, denoting by \( \nu_k(p, \theta) \) the \( k \)-th zero, we have 
\[ 
0 < \nu_1(p, \theta) < \nu_2(p, \theta) < \ldots < \nu_k(p, \theta) < \ldots \to +\infty.
\]
\end{remark}

\begin{proposition}\label{1: THM: Teorema Fondamentale}
Let \(\theta \geq 1\). For a given \( R > 0 \), the solutions \( h \in C^1[0,R] \) with \((r^{\theta-1}\phi_p(h'))'\in C^1[0,R]\) of the eigenvalue problem 

\[
\text{\bf (P.1)}\quad
\begin{cases} (r^{\theta-1}\phi_p(h'))' + \lambda r^{\theta-1}\phi_p(h) = 0 & \text{in } [0,R] \\ h'(0) = 0, \, h(R) = 0 
\end{cases}
\] 

are precisely:

\[
\begin{cases}
\lambda_k = \left(\frac{\nu_k(p,\theta)}{R}\right)^p \\
h_k(r) = c\varphi\left(\frac{\nu_k(p,\theta)}{R} r\right)
\end{cases}
\]
\vspace{0.2cm}

where \( c \in \mathbb{R} \) and  \(\varphi\) is the only solution of problem {\bf(P)}.
\end{proposition}

\begin{remark}\label{1: OSS: Unica Autofunzione Non Nulla}
It immediately follows from the definition that the function \( h_k \) has precisely \( k-1 \) zeros within the interval \( [0,R) \). In particular, the only solution of problem {\bf (P.1)} that never equals zero within \( [0,R) \) is \(\dyle h_1(r) = \varphi\left(\frac{\nu_1(p,\theta)}{R} r\right) \), up to a multiplicative constant.\\
\end{remark}

Let's fix a function \(d:\mathbb{R}^N\to [0,\infty)\) with the following properties: 

\begin{enumerate} 
\item \(d\in C(\mathbb{R}^N) \cap C^\infty(\mathbb{R}^N\setminus\{0\})\).
 
\item \(d(x)>0\) and \(d(x)=0\) if and only if \(x=0\).
 
\item \(d(\delta_\lambda(x))=\lambda d(x)\) (degree one homogeneity). 


\item Let \(Q\) denote the homogeneous dimension of \(\mathbb{R}^N\). Suppose that for every \(p\geq 2\), the following holds:

\[
\begin{cases}
\L_p d^{\frac{p-Q}{p-1}}=0 &\text{ in } \mathbb{R}^N\setminus\{0\} \text{ if } p\neq Q, \\
\L_Q(-\ln d)=0 &\text{ in } \mathbb{R}^N\setminus\{0\} \text{ if } p=Q.
\end{cases}
\]
\end{enumerate}

Define \(B^d_R\) as the set \(\label{NEW Gen: EQ: Definizione Palla}
B^d_R := \{x \in \mathbb{R}^N \,|\, d(x) \leq R\}. \)
Given a domain \(\Omega\) and a non-negative, locally integrable function \(\psi\), we define the space \( W^{1,p}_0(\Omega, \psi) \) as the closure of \( C^\infty_c(\Omega) \) with respect to the norm 
\begin{equation}\label{INTRO: EQ: Norma Sobolev}
\left(\int_\Omega |u|^p  |\nabla_\L d|^p\,\psi \, dx\right)^{\frac{1}{p}} + \left(\int_\Omega |\nabla_\L u|^p \psi \, dx\right)^{\frac{1}{p}}.
\end{equation}
We can now state our main result on Poincaré inequality.

\begin{theorem}[Weighted Poincaré inequality]\label{2: THM: Disuguaglianza di Poicnarè pesata}
Let \( p \geq 2 \), \(\al\geq 0\) and \( \theta \geq 1 \) be fixed. For every \( u \in W_0^{1,p}(B^d_R,|\nabla_\L d|^\al d^{\theta-Q }) \), the following inequalities hold:
\begin{equation}\label{2: EQ: THM: Weighted Poincaré Inequality}
\left(\frac{\nu_1(p,\theta)}{R}\right)^p\int_{B^d_R} \frac{|u|^p}{d^{Q-\theta}}|\nabla_\L d|^{\al+p}\,dx \leq \int_{B^d_R} \left|\nabla_\L u \cdot\frac{\nabla_\L d}{|\nabla_\L d|}\right|^p \frac{|\nabla_\L d|^\al}{d^{Q-\theta}}\,dx \leq \int_{B^d_R} \frac{|\nabla_\L u|^p}{d^{Q-\theta}}|\nabla_\L d |^\al\,dx \end{equation}
Moreover, the chain of inequalities is sharp since the function
\( u=\varphi\left(\frac{\nu_1(p,\theta)}{R}d\right) \in W_0^{1,p}(B^d_R,|\nabla_\L d|^\al d^{\theta-Q } ) \)
attains both equalities.
\end{theorem}

\begin{remark}
The sharp constant of the weighted Poincaré inequality depends on \( \theta \), \( p \), and \( R \), but does not depend on the topological dimension \(N\) or on the homogeneous dimension \(Q\).
\end{remark}

\begin{remark}
In the case \(\theta = 1\) and \(p = 2\), the constants can be explicitly calculated. Specifically, we have
\[ \nu_1(p, 1) = (p-1)^{\frac{1}{p}} \frac{\pi}{p \sin{\frac{\pi}{p}}}, \quad \nu_1(2, \theta) = z_{\frac{\theta-2}{2}} \]

where \(z_{\frac{\theta-2}{2}}\) is the first zero of the Bessel function of the first kind \(J_{\frac{\theta-2}{2}}\). For more details, see Section \ref{SEC EXAMPLES}.\\
\end{remark}

The previous theorem is essential for establishing the following two types of Hardy improvements.

\begin{theorem}[Hardy Improvement of type I]\label{4: THM: Hardy Improvement of type I}
For \(p \geq 2\) there exists a constant \(\lambda_p \in \left[\frac{1}{2^p}, \frac{p}{2^p}\right] \) such that, for every \(\al\geq 0\), \(\theta \in \mathbb{R}\) and for every \(u \in C^\infty_c(B^d_R \setminus \{0\})\), we have:

\[
\int_{B^d_R} \frac{|\nabla_\L u|^p}{d^{p(\theta - 1)}}|\nabla_\L d|^\al dx \geq \int_{B^d_R} \left|\nabla_\L u \cdot \frac{\nabla_\L d}{|\nabla_\L d|}\right|^p\frac{|\nabla_\L d|^\al}{d^{p(\theta-1)}} dx
\]
\begin{equation}\label{4: EQ: THM: Hardy Improvement of Type I}
\geq \left|\frac{Q - p\theta}{p}\right|^p \int_{B^d_R} \frac{|u|^p}{d^{p\theta}}|\nabla_\L d|^{\al+p} dx + \lambda_p \frac{\nu_1(p,p)^p}{R^p} \int_{B^d_R} \frac{|u|^p}{d^{p(\theta - 1)}}|\nabla_\L d |^{\al+p} dx. 
\end{equation}
The constant \(\lambda_p\) can be characterized as 
\(\dyle\lambda_p = \min_{x \in [1/2, 1]} \left\{ x^p + (1 - x)^{p-1}(x + p - 1) \right\}.\) In particular, for \(p = 2\), \(\lambda_p =1 \).
\end{theorem}

\begin{remark}
The constant in the last term on the right-hand side, \(\lambda_p \frac{\nu_1(p)^p}{R^p}\), depends exclusively on \(p\) and \(R\), and not on \(N\),\(Q\),\(\theta\) or \(\al\).
\end{remark}

\begin{theorem}[Hardy Improvement of Type II]\label{4: THM: Hardy Improvement of Type II}
Let \( p \geq 2 \), \(\al\geq 0\), \( \theta \in \mathbb{R} \), and \( z_0 \) be the first zero of the Bessel function \( J_0(r) \). Then, for every \( u \in C^\infty_c(B_R^d \setminus \{0\}) \), we have:

\[
\int_{B^d_R} \frac{|\nabla_\L u|^p}{d^{p(\theta - 1)}}|\nabla_\L d|^\al dx \geq \int_{B^d_R} \left|\nabla_\L u \cdot \frac{\nabla_\L d}{|\nabla_\L d|}\right|^p\frac{|\nabla_\L d|^\al}{d^{p(\theta-1)}} dx
\]
\begin{equation}\label{4: EQ: THM: Hardy improvement of type II}
\geq \left|\frac{Q - p\theta}{p}\right|^p \int_{B^d_R} \frac{|u|^p}{d^{p\theta}}|\nabla_\L d|^{\al+p} dx + \frac{2} {p} \left| \frac{Q-p\theta}{p} \right|^{p-2} \frac{z_0^2}{R^2} \int_{B^d_R} \frac{|u|^p}{d^{p\theta - 2}}|\nabla_\L d |^{\al+p} dx 
\end{equation}
\end{theorem}
\vspace{0.3cm}

\begin{remark}
It's noteworthy that for \( p = 2 \), the Type I and Type II improvements and their respective constants coincide.
\end{remark}



\section{Weighted Poincaré inequality}\label{SEC POINCARE}
{\bf Some remarks on the function \(d\) and integration by parts.}\\
Let \( Z \) be a vector field and \( w \) be a function, both assumed to be sufficiently regular. The following identities are straightforward to verify:
\[\dyle \Div_\L(wZ)=-\nabla_\L^*(wZ)=\nabla_\L w\cdot Z+w\Div_\L Z\]
\[\dyle \int_\Omega \Div_\L(Z)\,dx=\int_\Omega\Div(\sigma^TZ)\,dx=\int_{\partial \Omega} \sigma^TZ\cdot \nu \,dH_{N-1}=\int_{\partial \Omega} Z\cdot \sigma\nu \,dH_{N-1}=\int_{\partial \Omega}Z\cdot\nu_\L\,dH_{N-1}.\]

Here, \(\nu_\L=\sigma\nu\) and \(\nu\) is the outward unit normal vector.

Some very useful consequences of the assumptions made on the function \(d\) are as follows:

\begin{lemma}\label{NEW: LEMMA: p-Laplaciano di d}
For every \(p \geq 2\) and \(\al\geq 0\), we have:

\[
\L_p d=\frac{(Q-1)|\nabla_\L d|^p}{d}, \quad \nabla_\L |\nabla_\L d|\cdot \nabla_\L d =0 \quad \text{and} \quad \Div_\L \left(\frac{|\nabla_\L d|^{\al+p-2}\nabla_\L d}{d^{Q-1}}\right)=0 \text{ in } \mathbb{R}^N\setminus\{0\}.
\]
\end{lemma}

\begin{proof}
The third identity follows directly from the first two. For the first result, if \(p \neq Q\), then by hypothesis, we have:

\[
\begin{aligned}
0 &= \L_p d^{\frac{p-Q}{p-1}} 
= -\nabla_\L^* \left( \left| \frac{p-Q}{p-1} \right|^{p-2} \left( \frac{p-Q}{p-1} \right) d^{1-Q} |\nabla_\L d|^{p-2} \nabla_\L d \right) \\
&= \left| \frac{p-Q}{p-1} \right|^{p-2} \left( \frac{p-Q}{p-1} \right) \left\{ (1-Q)d^{-Q}|\nabla_\L d|^p + d^{1-Q}\L_p d \right\}.
\end{aligned}
\]
\vspace{0.1cm}

Thus, the result follows for the case \(p \neq Q\). Similarly, if \(p = Q\), then:

\[
\begin{aligned}
0 &= \L_Q (-\ln d) 
= -\nabla_\L^* \left( -\frac{1}{d^{Q-1}}|\nabla_\L d|^{Q-2}\nabla_\L d \right) 
= \frac{Q-1}{d^Q} |\nabla_\L d|^Q - \frac{1}{d^{Q-1}} \L_Q d.
\end{aligned}
\]

Finally, for fixed \(2 < p \neq Q\), we have:

\[
\begin{aligned}
\frac{(Q-1)|\nabla_\L d|^p}{d}& = \L_p d = \Div_\L( |\nabla_\L d|^{p-2}\nabla_\L d) = (p-2)|\nabla_\L d|^{p-3}\nabla_\L |\nabla_\L d|\cdot \nabla_\L d + |\nabla_\L d|^{p-2}\L_2 d\\
& = (p-2)|\nabla_\L d|^{p-3}\nabla_\L |\nabla_\L d|\cdot \nabla_\L d + \frac{(Q-1)|\nabla_\L d|^p}{d}.
\end{aligned}
\]

From this, we can conclude that \(\nabla_\L |\nabla_\L d|\cdot \nabla_\L d = 0\).
\end{proof}

\begin{lemma}
The set \(\{x\in\mathbb{R}^N \,|\, d(x)\leq R\}\) is compact for every \(R>0\).
\end{lemma}
\begin{proof}
Let's define the function \(\mathcal{N}(x)=\left(\sum_{j=1}^N |x_j|^{\frac{2}{\beta_j}}\right)^{\frac{1}{2}}\). It's clear that \(\mathcal{N}(x)\) is continuous on \(\mathbb{R}^N\), always positive, and equals zero only at \(x=0\). Furthermore, it's homogeneous of degree one with respect to the dilation \(\delta_\lambda\) and it's immediate to verify that the set \(\{x\in\mathbb{R}^N \,|\, \mathcal{N}(x)=1\}\) is compact.\\
Let's denote \(c:= \max\{H,{1}/{h} \}\), where 

\[ 
H:=\sup\{d(x) \,|\, \mathcal{N}(x)=1\}, \quad h:=\inf\{d(x) \,|\, \mathcal{N}(x)=1\}.
\]

Thanks to the homogeneity of \(d(x)\) and \(\mathcal{N}(x)\) under \(\delta_\lambda\), we have

\[ 
c^{-1}\mathcal{N}(x) \leq d(x) \leq c \mathcal{N}(x) \quad \forall x\in \mathbb{R}^N.
\]
In particular, the set \(\{x\in\mathbb{R}^N \,|\, d(x)\leq R\}\) is compact for every \(R>0\) since the set \(\{x\in\mathbb{R}^N \,|\, \mathcal{N}(x)\leq R\}\) is compact.\\ We explicitly observe that \(H<\infty\) and \(h>0\) because \(\{x\in\mathbb{R}^N \,|\, \mathcal{N}(x)=1\}\) is a compact set contained in \(\mathbb{R}^N\setminus\{0\}\), and \(d(x)\) is a continuous and positive function on \(\mathbb{R}^N\setminus\{0\}\).
\end{proof}

\begin{lemma}\label{NEW Gen: Lemma: Misura volume e superficie palla}
Let \(\alpha \geq 0\) and define \(\lambda_\alpha:=\int_{d\leq 1} |\nabla_\L d |^\alpha \, dx\). Then, we have:

\[
\int_{\{d\leq R\}} |\nabla_\L d|^\alpha \, dx = \lambda_\alpha R^Q,\qquad
\int_{\{d=R\}} \frac{|\nabla_\L d|^\alpha}{|\nabla d|} \, dH_{N-1} = Q\lambda_\alpha R^{Q-1}.
\]
\end{lemma}
We observe that \(\lambda_\alpha\) is well-defined since the set \(\{d \leq 1\}\) is compact and \(|\nabla_\L d|\) is a continuous function outside the origin that is \(0\)-\(\delta_\lambda\) homogeneous.

\begin{proof}
The second equation follows directly from the first one and the coarea formula by simply differentiating with respect to \(R\):

\[
\lambda_\alpha R^Q = \int_{d\leq R} |\nabla_\L d|^\alpha \, dx = \int_0^R \int_{d=s} \frac{|\nabla_\L d|^\alpha}{|\nabla d|} \, dH_{N-1} \, ds.
\]

For the first equation, a simple change of variables and the 0-homogeneity of \(|\nabla_\L d|\) gives:

\[
\int_{d\leq R}|\nabla_\L d|^\alpha \, dx = R^Q \int_{d\leq 1}|\nabla_\L d|^\alpha \, dx = \lambda_\alpha R^Q.
\]
\end{proof}

\vspace{0.3cm}

Using \(|E|\) to represent the Lebesgue measure of a set \(E\), it follows immediately from Lemma \refeq{NEW Gen: Lemma: Misura volume e superficie palla} that
\[ |B^d_R| = \lambda_0 R^Q \quad \text{and} \quad |\partial B^d_R| = Q\lambda_0 R^{Q-1}. \]

\begin{remark}\label{Remark smooth boundary}
We remark that \(\pa B^d_R\) is a smooth manifold of dimension \(N-1\). This holds true for almost every \(R>0\) according to Sard's lemma. The assertion then holds for  every \(R>0\) since \(\pa B_R^d\) is diffeomorphic to \(\pa B^d_1\) via the dilation \(\delta_\lambda\).
\end{remark}

\vspace{0.6cm}

We will frequently use the following lemma, which can be found in \cite{LD}.

\begin{lemma}[Proposition 2.1 \cite{LD}]
\label{2: Lemma: Proposizione identità Lp in LD}
Let \(D\subseteq\rN\) be a domain,
for every \(f,g\,\in L^p(D)\) and for every \(p\geq 2\) the following identity holds:
\begin{equation}\label{2: EQ: Idenità L^p In LD}
\norma{w(p,f,g)(f-g)}^2_{L^2(D)}
=\norma{f}^p_{L^p(D)}+(p-1)\norma{g}^p_{L^p(D)}
-p\left(\abs{g}^{p-2}g,\,f\right)_{L^2(D)}.
\end{equation}
The weight w is defined as
\begin{equation}\label{2: EQ: Definizione peso w Identità L^p In LD}
w(p,f,g)^2\coloneqq p(p-1)\int_0^1 
s\abs{sg+(1-s)f}^{p-2}\,ds.
\end{equation}
\end{lemma}

\vspace{0.5cm}

\begin{theorem}[Radial Poincaré Identity]\label{2: THM : Identità di Poincarè radiale}
Let \( p \geq 2 \), \( \theta \geq 1 \), \(\al\geq 0\) and let \( \varphi(\frac{\nu_1(p,\theta)}{R} r) \) denote the unique non-vanishing eigenfunction of problem {\bf (P.1)}. Given \(\dyle u \in C^\infty_c(B^d_R) \), define the functions:

\[
\begin{cases} 
\dyle f = \frac{\nabla_\L u}{d^{\frac{Q-\theta}{p}}} \cdot \frac{\nabla_\L d}{|\nabla_\L d|^{\frac{p-\al}{p}}} \\[15pt]
\dyle g = \frac{\nu_1(p,\theta)}{R}  \,\frac{\varphi'(\frac{\nu_1(p,\theta)}{R} d)}{\varphi(\frac{\nu_1(p,\theta)}{R} d)}  \,\frac{u}{d^{\frac{Q-\theta}{p}}}|\nabla_\L d|^{\frac{p+\al}{p}}
\end{cases}
\]

Then we have 

\begin{equation}\label{2: EQ: THM: Poincarè Identità}
\int_{B_R^d} \frac{|\nabla_\L d|^\al}{d^{Q-\theta}} \left\lvert\nabla_\L u \cdot \frac{\nabla_\L d}{|\nabla_\L d|}\right\rvert^p \, dx - \left(\frac{\nu_1(p,\theta)}{R}\right)^p \int_{B_R^d} \frac{|u|^p}{d^{Q-\theta}}|\nabla_\L d|^{p+\al} \, dx = \| w(p,f,g)(f-g) \|_{L^2(B_R^d)}^2.
\end{equation}
\end{theorem}

\begin{proof}
Given \( u \in C^\infty_c(B^d_R) \), in identity (\ref{2: EQ: Idenità L^p In LD}), we choose the functions \( f \) and \( g \) as follows:

\[
\begin{cases} 
\dyle f = \frac{\nabla_\L u}{d^{\frac{Q-\theta}{p}}} \cdot \frac{\nabla_\L d}{|\nabla_\L d|^{\frac{p-\al}{p}}} \\[15pt]
\dyle g = \frac{h'(d)}{h(d)} \cdot \frac{u}{d^{\frac{Q-\theta}{p}}}|\nabla_\L d|^{\frac{p+\al}{p}},
\end{cases}
\]

with \( h: [0,R] \to \mathbb{R} \) to be chosen later. To ensure rigor in all calculations, we impose the following assumptions on \( h \):

\[
\begin{cases} 
h \in C^1[0,R] \\
h \neq 0 \text{ in } [0,R) \\
r^{\theta-1} \phi_p(h') \in C^1 [0,R] \\
h'(0) = 0
\end{cases}
\]

It is easy to verify that the functions \( f \) and \( g \) chosen in this way belong to \( L^p(B^d_R) \). Let's compute the inner product:

\[ 
\begin{split}
&
(|g|^{p-2}g,f)_{L^2(B^d_R)} = \int_{B^d_R} \frac{d^{(\theta-1)}\phi_p(h')}{\phi_p(h)} |u|^{p-2}u \nabla_\L u \cdot \frac{|\nabla_\L d|^{\al+p-2}\nabla_\L d}{d^{Q-1}}\,dx \\[10pt]
&= \frac{1}{p} \int_{B^d_R}  \frac{d^{(\theta-1)}\phi_p(h')}{\phi_p(h)} \nabla_\L |u|^p \cdot \frac{|\nabla_\L d|^{\al+p-2}\nabla_\L d}{d^{Q-1}}\,dx= \quad(\text{see lemma \ref{NEW: LEMMA: p-Laplaciano di d} }) \\[10pt]
&
= -\frac{1}{p}\int_{B^d_R} |u|^p \frac{\nabla_\L (d^{(\theta-1)}\phi_p(h'))}{\phi_p(h)} \cdot \frac{|\nabla_\L d|^{\al+p-2}\nabla_\L d}{d^{Q-1}}\,dx 
+ \frac{1}{p}\int_{B^d_R} |u|^p d^{(\theta-1)}\phi_p(h') \frac{\nabla_\L \phi_p(h)}{\phi_p(h)^2} \cdot \frac{|\nabla_\L d|^{\al+p-2}\nabla_\L d}{d^{Q-1}}\,dx \\[10pt]
&
= -\frac{1}{p}\int_{B^d_R} \frac{|u|^p}{d^{Q-1}} \left[ \frac{1}{\phi_p(h)} \frac{d}{dr}(r^{(\theta-1)}\phi_p(h')) \right]_{r=d}|\nabla_\L d|^{\al+p}dx 
+ \frac{1}{p}\int_{B^d_R} \frac{|u|^p}{d^{Q-1}} \left[ \frac{(p-1) r^{(\theta-1)}\phi_p(h')|h|^{p-2}h'}{|h|^{2p-2}} \right]_{r=d}|\nabla_\L d|^{\al+p}dx \\[10pt]
&
= \frac{1}{p}\int_{B^d_R} \frac{|u|^p}{d^{Q-\theta}} \left[ (p-1)\frac{|h'|^p}{|h|^p} - \frac{1}{r^{(\theta-1)}\phi_p(h)} \frac{d}{dr} (r^{(\theta-1)}\phi_p(h')) \right]_{r=d}|\nabla_\L d|^{\al+p}dx.
\end{split}
\]

We observe that, to rigorously justify integration by parts, we must integrate over \(B^d_R \setminus B^d_\varepsilon\) and then take the limit as \(\varepsilon \to 0^+\). This is necessitated by the singularity of the integrand at the origin. In doing so, we find the following boundary term:

\[
\begin{split}
&\int_{\partial B^d_\varepsilon} \frac{|u|^p d^{\theta-1}\phi_p(h')(d)}{\phi_p(h)(d)}\,\, \frac{|\nabla_\L d|^{\al+p-2}\nabla_\L d}{d^{Q-1}} \cdot \nu_\L \,d\sigma 
\quad\quad \left( \nu_\L = \sigma \frac{\nabla d}{|\nabla d|} = \frac{\nabla_\L d}{|\nabla d|}\right) \\[10pt]
&=
\frac{\varepsilon^{\theta-1}}{\varepsilon^{Q-1}} \frac{\phi_p(h')(\varepsilon)}{\phi_p(h)(\varepsilon)} \int_{d=\varepsilon} |u|^p \frac{|\nabla_\L d|^{\al+p}}{|\nabla d|} \,d\sigma \leq C_{p,\al} \frac{\varepsilon^{\theta-1}\phi_p(h')(\varepsilon)}{\phi_p(h)(\varepsilon)} \to 0 \quad \text{as} \quad \varepsilon \to 0^+.
\end{split}
\]

It's noteworthy that for \(\theta=1\), the condition \(h'(0)=0\) becomes necessary in this context, whereas for \(\theta > 1\), it is a direct consequence of lemma \ref{1: Lemma: Regolarità in 0}.

Returning to our choice of \( f \) and \( g \), we have

\[
||w(p,f,g)(f-g)||_{L^2(B^d_R)}^2 = \int_{B^d_R} \frac{|\nabla_\L d|^{\al}}{d^{Q-\theta}} \left\lvert \nabla_\L u \cdot \frac{\nabla_\L d}{|\nabla_\L d|} \right\rvert^p \, dx + \int_{B^d_R} \frac{|u|^p}{d^{Q-\theta}} \left( \frac{1}{r^{\theta-1}\phi_p(h)} \frac{d}{dr} (r^{\theta-1}\phi_p(h')) \right)_{r=d} |\nabla_\L d|^{\al+p} dx.
\]

If we choose the function \( h \) in such a way that 

\[
\frac{1}{r^{\theta-1}\phi_p(h)}  \left(r^{\theta-1}\phi_p(h')\right)' = -\lambda,
\]

for some \( \lambda > 0 \), then we would have 

\[
||w(p,f,g)(f-g)||_{L^2}^2 = \int_{B^d_R} \frac{|\nabla_\L d|^\al}{d^{Q-\theta}} \left\lvert \nabla_\L u \cdot \frac{\nabla_\L d}{|\nabla_\L d|} \right\rvert^p \, dx - \lambda \int_{B^d_R} \frac{|u|^p}{d^{Q-\theta}}|\nabla_\L d|^{\al+p} \, dx.
\]

Therefore, we aim to find the largest constant \( \lambda \) such that the problem 

\[
\text{\bf(E)} \quad 
\begin{cases}
(r^{\theta-1}\phi_p(h'))' + \lambda r^{\theta-1}\phi_p(h) = 0& \text{ in } [0,R] \\ h'(0) = 0 \\ h \neq 0& \text{ in } [0,R) \end{cases}
\]

admits a solution \( h \in C^1[0,R] \), \( r^{\theta-1}\phi_p(h') \in C^1[0,R] \).\\

According to remark \ref{1: OSS: Unica Autofunzione Non Nulla}, the pair \(\lambda = \left(\frac{{\nu_1(p,\theta)}}{R}\right)^p \) and \( h = \varphi\left(\frac{{\nu_1(p,\theta)}}{R} d\right) \) constitutes a nontrivial solution to the problem {\bf (E)} (the unique solution \( h \neq 0 \) satisfying the initial conditions \( h'(0) = 0 \) and \( h(R) = 0 \)). We aim to demonstrate that this indeed constitutes the desired solution.

Let's define \( \mu := h(R) < +\infty \). Our goal is to show that \( \mu = 0 \).\\

Suppose \( \mu \neq 0 \) (\( h(0) \neq 0 \), otherwise \( h \equiv 0 \)). The function \(\dyle \tilde{k}(r) = \frac{1}{h(0)} h\left(\frac{r}{\lambda^{1/p}}\right) \) solves 

\[
\begin{cases} 
(r^{\theta-1}\phi_p(\tilde{k}))'+r^{\theta-1}\phi_p(\tilde{k})=0 \text{ in } [0,\lambda^{1/p}R] \\ 
\tilde{k}'(0)=0, \quad \tilde{k}(0)=1 
\end{cases}
\]

By the uniqueness of problem {\bf (P)}, we obtain \(\dyle h(r) = h(0)\varphi(\lambda^{1/p}r) \). Since \( h \neq 0 \) in \( [0,R) \) and \( h(R) = \mu \neq 0 \), then \( h \neq 0 \) in \( [0,R] \). This translates to \(\dyle \varphi(\lambda^{1/p}r) \neq 0 \) in \( [0,R] \), implying: 

\[ 
\lambda < \left(\frac{{\nu_1(p,\theta)}}{R}\right)^p.
\]

This shows that the largest \(\lambda\) for which problem {\bf(E)} has a solution is obtained when \( \mu = 0 \). This concludes the proof of the Theorem.
\end{proof}

\begin{remark}\label{2: OSS: Altra Classe di Funzioni per THM ID POINCARè}
All calculations in the proof of Theorem \ref{2: THM : Identità di Poincarè radiale} still hold if \( u \) takes the form 

\[ 
u(x) = \varphi\left(\frac{\nu_1(p,\theta)}{R} d(x)\right) v(x),\quad \text{where } v \in C^1(\overline{B_R^d}).
\] 

Indeed, in this case, the functions \( f \) and \( g \) belong to \( L^p \), 
\[
\begin{cases} 
\dyle f= \frac{\nabla_\L u}{d^{\frac{N-\theta}{p}}}\cdot \frac{\nabla_\L d}{|\nabla_\L d|^{\frac{p-\al}{p}}} \in L^p(B^d_R)\\[15pt] 
\dyle g= \frac{\nu_1(p,\theta)}{R} \varphi'\left(\frac{\nu_1(p,\theta)}{R} d(x)\right)\, \frac{v(x)}{d^{\frac{N-\theta}{p}}}\,|\nabla_\L d|^{\frac{p+\al}{p}} \in L^p(B^d_R). \end{cases}
\]
Moreover, in the integration by parts, the boundary term remains zero. Therefore, for functions \( u(x) \) of this form, Poincaré's identity (\ref{2: EQ: THM: Poincarè Identità}) holds.
\end{remark}
\vspace{0.5cm}


\begin{proof}[Proof of Theorem \ref{2: THM: Disuguaglianza di Poicnarè pesata}]
If \( u \in C^\infty_c(B^d_R) \), then the chain of inequalities is an obvious consequence of identity (\ref{2: EQ: THM: Poincarè Identità}) and the Cauchy-Schwarz inequality. If \( u \in W_0^{1,p}(B^d_R,|\nabla_\L d|^\al d^{\theta-Q } ) \), then we can find a sequence \( u_n \in C^\infty_c(B^d_R) \) converging to \( u \) with respect to the norm (\ref{INTRO: EQ: Norma Sobolev}). By passing to the limit as \( n \to \infty \) in (\ref{2: EQ: THM: Weighted Poincaré Inequality}), the thesis follows. 

%

From remark \ref{2: OSS: Altra Classe di Funzioni per THM ID POINCARè}, we see that for the function 
\(u=\varphi\left(\frac{\nu_1(p,\theta)}{R} d\right) \)
the Poincaré identity (\ref{2: EQ: THM: Poincarè Identità}) holds. In particular, in this case \( f-g \equiv 0 \), thus

\begin{equation}\label{2: EQ: DIMO: Identità Varphi} \frac{\nu_1(p,\theta)^p}{R^p} \int_{B^d_R} \frac{|u|^p}{d^{Q-\theta}}|\nabla_\L d|^{\al+p}\, dx=\int_{B^d_R} \left|\nabla_\L u \cdot \frac{\nabla_\L d}{|\nabla_\L d|}\right|^p \frac{|\nabla_\L d|^\al}{d^{Q-\theta}}\,dx \quad \text{for } u=\varphi\left(\frac{\nu_1(p,\theta)}{R} d\right).
\end{equation} 

Since \( u=\varphi\left(\frac{\nu_1(p,\theta)}{R} d\right) \) is a \( C^1(\overline{B^d_R}) \) function that vanishes on the boundary and \(\partial B^d_R\) is regular (Remark \ref{Remark smooth boundary}), it easily follows that \(u \in W_0^{1,p}(B^d_R,|\nabla_\L d|^\al d^{\theta-Q })\). Furthermore, \( u \) is \(d\)-radial, \( |\nabla_\L u|^p=\left|\nabla_\L u \cdot \frac{\nabla_\L d}{|\nabla_\L d|}\right|^p \), this completes the proof of the theorem.
\end{proof}





\subsection{Some Special Cases}\label{SEC EXAMPLES}
In this section, we aim to highlight that the cases where \( \theta = 1 \) and \( p = 2 \) are fully solvable.

\subsubsection{The case \( \theta = 1 \).}

\begin{proposition}
For every \( p \geq 2 \), \(\al\geq 0\) and for every \( u \in W_0^{1,p}(B^d_R,|\nabla_\L d|^\al d^{1-Q } ) \), it holds:

\[
\frac{(p-1)}{R^p} \left( \frac{\pi}{p\sin{\frac{\pi}{p}}} \right)^p \int_{B^d_R} \frac{|u|^p}{d^{Q-1}}|\nabla_\L d|^{\al+p}\,dx \leq \int_{B^d_R} \left| \nabla_\L u \cdot\frac{ \nabla_\L d}{|\nabla_\L d|} \right|^p \frac{|\nabla_\L d|^\al}{d^{Q-1}}\, dx \leq \int_{B^d_R} \frac{|\nabla_\L u|^p}{d^{Q-1}}|\nabla_\L d|^\al\, dx 
\]

Moreover, the chain of inequalities is sharp.
\end{proposition}

\begin{proof}
The proposition immediately follows if we demonstrate that the first zero \( \nu_1(p,1) \) of the function \( \varphi(r) \) is \(  (p-1)^{1/p} \frac{\pi}{p\sin{\frac{\pi}{p}}} \). Recall that in this case, \( \varphi(r) \) is defined as the unique solution of the problem:
 
\[
\begin{cases} 
(|\varphi'|^{p-2}\varphi')'+|\varphi|^{p-2}\varphi=0, & r \geq 0 \\ 
\varphi'(0)=0, \quad \varphi(0)=1 
\end{cases}
\]

For \( p=2 \), the problem is easily solvable, and the unique solution is \( \varphi(r)=\cos(r) \). As we aim to demonstrate, the first zero occurs at \(\dyle r=\frac{\pi}{2} \).\\
For the general case, let \( \nu_1(p,1)=\nu_1(p) \) denote the first zero of \( \varphi(r) \). By definition, \( \varphi(r)>0 \) in \( [0,\nu_1(p)) \).
Integrating from 0 to \( r \), we find:
\(\dyle |\varphi'|^{p-2}\varphi'(r)=-\int_0^r |\varphi|^{p-2}\varphi(s)\,ds < 0 \) for \( r \in (0,\nu_1(p)). \)
Summarizing, we have:

\[
\begin{cases} 
\varphi(r) > 0   &\text{in} \quad (0,\nu_1(p)) \\
\varphi'(r) < 0  &\text{in} \quad (0,\nu_1(p)) 
\end{cases}
\]

Multiplying the differential equation by \( \varphi' \), we get:
\(\dyle (|\varphi'|^{p-2}\varphi')'\varphi' +\frac{1}{p} (|\varphi|^p)'=0. \)
After some manipulation, this transforms into:

\[ 
\left( \frac{p-1}{p} |\varphi'|^p+ \frac{1}{p} |\varphi|^p \right)'=0. 
\]

Integrating from 0 to \( r \geq 0 \), we find:

\[ 
(p-1)|\varphi'(r)|^{p}+|\varphi(r)|^p=1 \quad \forall r \geq 0. 
\]

In particular, on the interval \( (0,\nu_1(p)) \), we have:

\[ \varphi'(r)=-\left[\frac{1-\varphi(r)^p}{p-1}\right]^{1/p}  \quad \text{for } r \in (0,\nu_1(p)). 
\]

We can conclude:
\[
\begin{aligned}
\nu_1(p) &= -\int_0^{\nu_1(p)} (-1)dx 
= -\int_1^0 \frac{(p-1)^{1/p}}{(1-\varphi^p)^{1/p}} d\varphi 
= (p-1)^{1/p} \int_0^1 \frac{1}{(1-\varphi^p)^{1/p}} d\varphi \quad (\text{let } t=\varphi^p) \\[10pt]
&= \frac{(p-1)^{1/p}}{p} \int_0^1 t^{\frac{1}{p}-1}(1-t)^{-\frac{1}{p}}dt 
= \frac{(p-1)^{1/p}}{p} B\left(\frac{1}{p},1-\frac{1}{p}\right) 
= (p-1)^{1/p}\frac{ \pi}{p\sin\left(\frac{\pi}{p}\right)}.
\end{aligned}
\]

\( B(x,y) \) denotes the Euler beta function, and, in the last identity, we used Euler's reflection formula
\[ B(x,1-x)=\frac{\pi}{\sin(\pi x)} \quad\text{ if } x \notin \mathbb{Z}. \]
This concludes the proof of the proposition.
\end{proof}

\subsubsection{The case \(p=2\).}
Choosing \( p=2 \) leads to another significant example. In this case, the function \( \varphi(r) \) is defined as the unique solution of the problem:

\[
\begin{cases} 
(r^{\theta-1}\varphi')'+r^{\theta-1}\varphi=0, & r\geq0 \\
\varphi(0)=1, \quad \varphi'(0)=0
\end{cases}
\]

A simple verification shows that the solution to this problem is:

\[ 
\varphi(r)=\Gamma\left(\theta/2\right)\sum_{m=0}^{\infty} \frac{(-1)^m}{m!\Gamma(m+\frac{\theta}{2})} \left(\frac{r}{2}\right)^{2m} = \Gamma\left(\theta/2\right) \left(\frac{2}{r}\right)^{\frac{\theta-2}{2}} J_{\frac{\theta-2}{2}}(r)
\]

Where \(\dyle J_{\frac{\theta-2}{2}} \) is the Bessel function of the first kind. Hence, the first zero of \( \varphi \), \( \nu_1(2,\theta) \), is simply the first zero of the function \(\dyle J_{\frac{\theta-2}{2}} \).\\

We can summarize in the following proposition.

\begin{proposition} 
Let \( \theta \geq 1 \), \(\al\geq 0\) and let \( z_{\frac{\theta-2}{2}} \) be the first zero of the Bessel function \(\dyle J_{\frac{\theta-2}{2}} \). For every \( u \in W_0^{1,2}(B^d_R,|\nabla_\L d|^\al d^{\theta-Q } ) \) it holds:
%
%
%
%
\[ 
\frac{z_{\frac{\theta-2}{2}}^2}{R^2} \int_{B^d_R} \frac{|u|^2}{d^{Q-\theta}}|\nabla_\L d|^{2+\al} \,dx \leq \int_{B^d_R} \left|\nabla_\L u \cdot \frac{\nabla_\L d}{|\nabla_\L d|}\right|^2 \frac{|\nabla_\L d|^\al}{d^{Q-\theta}} \,dx 
\leq \int_{B^d_R} \frac{|\nabla_\L u|^2}{d^{Q-\theta}}|\nabla_\L d |^\al \,dx
\]
Moreover, this chain of inequalities is sharp.

\end{proposition}


\section{Improved Hardy inequality}\label{SEC Hardy IMPROVEMENT}

\begin{theorem}\label{4: THM: Stime dal basso}
Let \( D \subset \mathbb{R}^N \) be a domain, and let \( p \geq 2 \). Then:

\begin{itemize}
  \item \textbf{First Type Estimate.} \\
  There exists a constant \(  \lambda_p \in \left[\frac{1}{2^p}, \frac{p}{2^p}\right] \) such that for every \( f, g \in L^p(D) \):
  \begin{equation}\label{4: EQ: THM: Stima del primo tipo}
  ||f||^p_{L^p(D)} + (p-1) ||g||^p_{L^p(D)} - p (|g|^{p-2}g, f)_{L^2(D)} = ||w(p, f, g) (f-g)||^2_{L^2(D)} \geq \lambda_p ||f-g||^p_{L^p(D)}
  \end{equation}

  The constant \(\lambda_p\) can be characterized as 
  \(\dyle
  \lambda_p = \min_{x \in [1/2, 1]} \left\{ x^p + (1 - x)^{p-1}(x + p - 1) \right\}.
  \) 
  In particular, for \(p = 2\), \(\lambda_p = 1\).
 
  \item \textbf{Second Type Estimate.} \\
  For every \( f, g \in L^p(D) \):
  \begin{equation}\label{4: EQ: THM: Stima del secondo tipo}
  ||f||^p_{L^p(D)} + (p-1) ||g||^p_{L^p(D)} - p (|g|^{p-2}g, f)_{L^2(D)} = ||w(p, f, g) (f-g)||^2_{L^2(D)} \geq \frac{p}{2} \int_D |g|^{p-2} (f-g)^2 \, dx
  \end{equation}
\end{itemize}
\end{theorem}

The theorem is an obvious corollary of (\ref{2: EQ: Idenità L^p In LD}) and the following two propositions, which are proven in Appendix \ref{Appendice B}.

\begin{proposition}\label{STIME: PROP: STIMA DEL PRIMO TIPO}
For any \(p \geq 2\), there exists a constant \(\lambda_p \in \left[\frac{1}{2^p}, \frac{p}{2^{p-1}}\right]\) such that, for any \(x, y \in \mathbb{R}\),
\[
w^2(p, x, y)
=
p(p-1) \int_0^1 s |sy + (1-s)x|^{p-2}\, ds
\geq
\lambda_p |x-y|^{p-2}.
\]
\end{proposition}

\begin{proposition}\label{STIME: PROP: STIMA DEL SECONDO TIPO}
For any \(p \geq 2\) and for any \(x, y \in \mathbb{R}\),
\[
w^2(p, x, y)
=
p(p-1) \int_0^1 s |sy + (1-s)x|^{p-2}\, ds
\geq
\frac{p}{2} |y|^{p-2}.
\]
\end{proposition}
\vspace{0.5cm}


\begin{proof}[Proof of Theorem \ref{4: THM: Hardy Improvement of type I}]
A simple calculation in the spirit of \cite{LD} shows that if we choose in (\ref{2: EQ: Idenità L^p In LD}):

\[
\begin{cases}
\dyle
f = \frac{\nabla_\L u}{d^{\theta-1}}\cdot \frac{\nabla_\L d}{|\nabla_\L d|^{\frac{p-\al}{p}}} \in L^p(B_R^d)\\[15pt] 
\dyle
g = -\left(\frac{Q - p\theta}{p}\right)\, \frac{u}{d^\theta}|\nabla_\L d|^{\frac{p+\al}{p}} \in L^p(B_R^d) 
\end{cases}
\]

then we have:

\[ 
\int_{B_R^d} \left|\nabla_\L u \cdot \frac{\nabla_\L d}{|\nabla_\L d|}\right|^p\frac{|\nabla_\L d|^\al}{d^{p(\theta-1)}} dx - \left| \frac{Q - p\theta}{p} \right|^p \int_{B_R^d} \frac{|u|^p}{d^{p\theta}}|\nabla_\L d|^{\al+p} dx = \left\| w(p,f,g)(f-g) \right\|_{L^2(B_R^d)}^2. 
\]

Using the lower bounds provided by (\ref{4: EQ: THM: Stima del primo tipo}), we have:
\[
\begin{split}
&\left\| w(p,f,g)(f-g) \right\|_{L^2} \geq \lambda_p \int_{B_R^d} \left| \frac{\nabla_\L u}{d^{\theta-1}}\cdot \frac{\nabla_\L d}{|\nabla_\L d|^{\frac{p-\al}{p}}} + \left(\frac{Q - p\theta}{p}\right)\, \frac{u}{d^\theta}|\nabla_\L d|^{\frac{p+\al}{p}} \right|^p dx =\\
\end{split}
\]
\[
=\lambda_p \int_{B_R^d} \frac{|\nabla_\L d|^\al}{d^{Q-p}} \left| \nabla_\L \left(u d^{\frac{Q-p\theta}{p}}\right) \cdot \frac{\nabla_\L d}{| \nabla_\L d|} \right|^p dx.
\]

Now, we can use Poincaré's inequality (\ref{2: EQ: THM: Weighted Poincaré Inequality}), obtaining

\[ 
\left\| w(p,f,g)(f-g) \right\|_{L^2}
\geq \lambda_p \frac{\nu_1(p,p)^p}{R^p} \int_{B_R^d} \frac{|u|^p d^{Q-p \theta}}{d^{Q-p}}|\nabla_\L d|^{\al+p} dx = \lambda_p \frac{\nu_1(p,p)^p}{R^p} \int_{B_R^d} \frac{|u|^p}{d^{p(\theta-1)}} |\nabla_\L d|^{\al+p}dx. 
\]

This completes the proof of the theorem.
\end{proof}


\begin{proof}[Proof of Theorem \ref{4: THM: Hardy Improvement of Type II}]
Let's proceed as in the proof of Theorem \ref{4: THM: Hardy Improvement of type I}. With the choice 

\[
\begin{cases}
\dyle
f = \frac{\nabla_\L u}{d^{\theta-1}}\cdot \frac{\nabla_\L d}{|\nabla_\L d|^{\frac{p-\al}{p}}} \in L^p(B_R^d)\\[15pt] 
\dyle
g = -\left(\frac{Q - p\theta}{p}\right)\, \frac{u}{d^\theta}|\nabla_\L d|^{\frac{p+\al}{p}} \in L^p(B_R^d) 
\end{cases}
\]

in (\ref{2: EQ: Idenità L^p In LD}) and from (\ref{4: EQ: THM: Stima del secondo tipo}), we have 

\[
\begin{aligned}
&\int_{B_R^d} \left|\nabla_\L u \cdot \frac{\nabla_\L d}{|\nabla_\L d|}\right|^p\frac{|\nabla_\L d|^\al}{d^{p(\theta-1)}} dx - \left| \frac{Q - p\theta}{p} \right|^p \int_{B_R^d} \frac{|u|^p}{d^{p\theta}}|\nabla_\L d|^{\al+p} dx = \left\| w(p,f,g)(f-g) \right\|_{L^2(B_R^d)}^2\\[10pt]
&
\geq \frac{p}{2}\int_{B_R^d}|g|^{p-2}(f-g)^2 dx 
= \frac{p}{2} \left| \frac{Q - p\theta}{p} \right|^{p-2} \int_{B_R^d} \frac{|u|^{p-2}}{d^{p\theta-2}} \left( \nabla_\L u \cdot \frac{\nabla_\L d}{|\nabla_\L d|} + \frac{Q - p\theta}{p} \frac{u}{d}|\nabla_\L d| \right)^2 |\nabla_\L d|^{\al+p-2} dx \\[10pt]
&= \frac{p}{2} \left| \frac{Q - p\theta}{p} \right|^{p-2} \int_{B_R^d} \frac{|\nabla_\L d|^{\al+p-2}}{d^{Q-2}}\left| u d^{\frac{Q-p\theta}{p}} \right|^{p-2} \left( \nabla_\L \left( u d^{\frac{Q-p\theta}{p}} \right) \cdot \frac{\nabla_\L d}{|\nabla_\L d|} \right)^2 dx \\[10pt]
&= \frac{2}{p} \left| \frac{Q - p\theta}{p} \right|^{p-2} \int_{B_R^d} \frac{|\nabla_\L d|^{\al+p-2}}{d^{Q-2}} \left[ \nabla_\L \left| u d^{\frac{Q-p\theta}{p}} \right|^{\frac{p}{2}} \cdot \frac{\nabla_\L d}{|\nabla_\L d|} \right]^2 dx \\[10pt]
&\geq\frac{2}{p} \left| \frac{Q - p\theta}{p} \right|^{p-2} \frac{z_0^2}{R^2} \int_{B_R^d} \frac{|u|^p}{d^{p\theta-2}} |\nabla_\L d |^{\al+p}dx.
\end{aligned}
\]

Here, we first use the fact that \(|v|^{p-2}|\nabla_\L v\cdot \frac{\nabla_\L d}{|\nabla_\L d|}|^2 = \frac{4}{p^2}|\nabla_\L|v|^{\frac{p}{2}}\cdot\frac{\nabla_\L d}{|\nabla_\L d|}|^2\), and then the Poincaré inequality (\ref{2: EQ: THM: Weighted Poincaré Inequality}).
\end{proof}

\section{Poincaré inequalities and Hardy improvements for some sub-elliptic operators}\label{SEC: Applicazioni particolari}
In this section, we will apply the previous result to some specific operators.

\subsection{Euclidean Laplacian}
In \(\mathbb{R}^N\), let's consider the vector field \(\nabla_\mathcal{L} = \nabla\) and the corresponding p-Laplacian operator \(\Delta_p\). Let \(d(x) = |x|\) be the Euclidean distance from the origin, and let \(\delta_\lambda\) be the family of dilations \(\delta_\lambda(x) = (\lambda x_1, \ldots, \lambda x_N)\). In this case, where \(Q = N\), it's immediate to verify that both \(\nabla\) and \(|x|\) are homogeneous of degree one with respect to \(\delta_\lambda\).
Furthermore, for \(p \geq 2\), 

\[
\begin{cases} 
\Delta_p |x|^{\frac{p-N}{p-1}}=0 &\text{ in } \mathbb{R}^N \setminus \{0\},  \text{ if } p \neq N \\
\Delta_N(-\ln|x|)=0 &\text{ in } \mathbb{R}^N \setminus \{0\}, \text{ if } p = N.
\end{cases}
\]
Additionally, note that \(|\nabla d|=|\nabla |x|| = 1\).\\

The results from Section \ref{SEC: Main Results} are applicable in this context, we have:

\begin{theorem}[Weighted Poincaré inequality]\label{6Euc: THM: Disuguaglianza di Poicnarè pesata}
Let \( p \geq 2 \) and \( \theta \geq 1 \) be fixed. For every \( u \in W_0^{1,p}(B_R, |x|^{\theta-N}) \), the following inequalities hold:
\begin{equation}\label{6Euc: EQ: THM: Weighted Poincaré Inequality}
\left(\frac{\nu_1(p,\theta)}{R}\right)^p\int_{B_R} \frac{|u|^p}{|x|^{N-\theta}} \,dx \leq \int_{B_R} \left|\nabla u \cdot\frac{x}{|x|}\right|^p \frac{1}{|x|^{N-\theta}}\,dx \leq \int_{B_R} \frac{|\nabla u|^p}{|x|^{N-\theta}} \,dx
\end{equation}
Moreover, the chain of inequalities is sharp since the function
\( u = \varphi\left(\frac{\nu_1(p,\theta)}{R}|x|\right) \in W_0^{1,p}(B_R, |x|^{\theta-N}) \)
attains both equalities.
\end{theorem}

An immediate corollary of Theorem \ref{6Euc: THM: Disuguaglianza di Poicnarè pesata} is the following:
\begin{corollary}
The first eigenvalue of the \(p\)-Laplacian operator in \(B_R \subset \mathbb{R}^N\) is \(\lambda_1 = \left(\frac{\nu_1(p,\theta)}{R}\right)^p\), and the corresponding eigenfunction is \(\varphi\left(\frac{\nu_1(p,\theta)}{R}|x|\right)\).
\end{corollary}
\begin{proof}
Setting \(\theta = N\) in (\ref{6Euc: EQ: THM: Weighted Poincaré Inequality}) and recalling the variational characterization of the optimal Poincaré constant, the statement is straightforward.
\end{proof}

\begin{theorem}[Hardy Improvement of Type I]\label{6Euc: THM: Hardy Improvement of type I}
Let \(p \geq 2\), \(\theta \in \mathbb{R}\), and \(\lambda_p\) be the constant defined in Theorem \ref{4: THM: Stime dal basso}. For every \(u \in C^\infty_c(B_R \setminus \{0\})\), we have:
\[
\int_{B_R} \frac{|\nabla u|^p}{|x|^{p(\theta - 1)}} \, dx \geq \int_{B_R} \left|\nabla u \cdot \frac{x}{|x|}\right|^p\frac{1}{|x|^{p(\theta-1)}} \, dx
\]
\begin{equation}\label{6Euc: EQ: THM: Hardy Improvement of Type I}
\geq \left|\frac{Q - p\theta}{p}\right|^p \int_{B_R} \frac{|u|^p}{|x|^{p\theta}} \, dx + \lambda_p \frac{\nu_1(p,p)^p}{R^p} \int_{B_R} \frac{|u|^p}{|x|^{p(\theta - 1)}} \, dx.
\end{equation}

\end{theorem}

\vspace{0.2cm}

\begin{theorem}[Hardy Improvement of Type II]\label{6Euc: THM: Hardy Improvement of Type II}
Let \( p \geq 2 \), \( \theta \in \mathbb{R} \), and \( z_0 \) be the first zero of the Bessel function \( J_0(r) \). Then, for every \( u \in C^\infty_c(B_R \setminus \{0\}) \), we have:
\[
\int_{B_R} \frac{|\nabla u|^p}{|x|^{p(\theta - 1)}} \, dx \geq \int_{B_R} \left|\nabla u \cdot \frac{x}{|x|}\right|^p\frac{1}{|x|^{p(\theta-1)}} \, dx
\]
\begin{equation}\label{6Euc: EQ: THM: Hardy improvement of type II}
\geq \left|\frac{N - p\theta}{p}\right|^p \int_{B_R} \frac{|u|^p}{|x|^{p\theta}} \, dx + \frac{2} {p} \left| \frac{N-p\theta}{p} \right|^{p-2} \frac{z_0^2}{R^2} \int_{B_R} \frac{|u|^p}{|x|^{p\theta - 2}} \, dx.
\end{equation}
\end{theorem}
\vspace{1cm}

\begin{remark}
It is interesting to note that for \( p=2 \), the Type I and Type II improvements coincide with the weighted version of the Brezis-Vázquez inequality. Specifically, we have:

\[
\int_{B_R} \frac{|\nabla u|^2}{|x|^{2(\theta - 1)}} \, dx \geq \int_{B_R} \left|\nabla u \cdot \frac{x}{|x|}\right|^2\frac{1}{|x|^{2(\theta-1)}} \, dx
\geq \left(\frac{N - 2\theta}{2}\right)^2 \int_{B_R} \frac{|u|^2}{|x|^{2\theta}} \, dx + \frac{z_0^2}{R^2} \int_{B_R} \frac{|u|^2}{|x|^{2(\theta - 1)}} \, dx.
\]

For \(\theta = 1\), this matches exactly with equation (\ref{INTRO: EQ: Brezis-Vazquez}).
\end{remark}

\subsection{Baouendi-Grushin Operator}
Let \(\mathbb{R}^N\) be split into \((x, y) \in \mathbb{R}^n \times \mathbb{R}^k\). Fix \(\gamma \geq 0\) and consider the vector field \(\nabla_\gamma = (\nabla_x, (1+\gamma)|x|^\gamma \nabla_y)\). The corresponding linear operator \(\mathcal{L} = \mathcal{L}_2\) is the so-called Baouendi-Grushin operator \(\Delta_\gamma = \Delta_x + (1+\gamma)^2 |x|^{2\gamma} \Delta_y\), while \(\mathcal{L}_p\) is given by \(\mathcal{L}_p = \operatorname{div}_\mathcal{L} (|\nabla_\gamma u|^{p-2} \nabla_\gamma u)\).\\

Note that for \(k = 0\) or \(\gamma = 0\), \(\mathcal{L}\) and \(\mathcal{L}_p\) correspond to the usual Euclidean Laplacian and p-Laplacian, respectively.\\

Let \(\rho(x,y)\) be the following distance from the origin in \(\mathbb{R}^N\):
\[
\rho(x, y) = \left( |x|^{2(1+\gamma)} + |y|^2 \right)^{\frac{1}{2(1+\gamma)}}.
\]

The vector field \(\nabla_\gamma\) and the function \(\rho(x, y)\) are homogeneous of degree one with respect to the family of dilations \(\delta_\lambda = (\lambda x, \lambda^{1+\gamma} y)\) for \(\lambda > 0\). Let \(Q = n + (1+\gamma)k\) be the so-called homogeneous dimension, a straightforward calculation shows that:

\[
\begin{cases} 
\mathcal{L}_p \rho^{\frac{p-Q}{p-1}} = 0 & \text{in } \mathbb{R}^N \setminus \{0\}\,\,\, \text{ if } p \neq Q \\
\mathcal{L}_Q(-\ln \rho) = 0 & \text{in } \mathbb{R}^N \setminus \{0\}\,\,\, \text{ if } p = Q
\end{cases}
\]

Finally, we observe that \(|\nabla_\gamma \rho|^2 = \frac{|x|^{2\gamma}}{\rho^{2\gamma}} \leq 1\).\\

All the results from Section \ref{SEC: Main Results} hold, and thus we can summarize them as follows.

\begin{theorem}[Weighted Poincaré inequality]\label{6Grushin: THM: Disuguaglianza di Poicnarè pesata}
Let \( p \geq 2 \), \(\al\geq 0\) and \( \theta \geq 1 \) be fixed. For every \( u \in W_0^{1,p}(B^\rho_R, |\nabla_\gamma \rho|^\al \rho^{\theta-Q}) \), the following inequalities hold:
\begin{equation}\label{6Grushin: EQ: THM: Weighted Poincaré Inequality}
\left(\frac{\nu_1(p,\theta)}{R}\right)^p\int_{B_R^\rho} \frac{|u|^p}{\rho^{Q-\theta}}|\nabla_\gamma \rho|^{\alpha+p}\,dx 
\leq \int_{B_R^\rho} \left|\nabla_\gamma u \cdot\frac{\nabla_\gamma \rho}{|\nabla_\gamma \rho|}\right|^p \frac{|\nabla_\gamma \rho|^\alpha}{\rho^{Q-\theta}}\,dx
\leq \int_{B_R^\rho} \frac{|\nabla_\gamma u|^p}{\rho^{Q-\theta}}|\nabla_\gamma \rho |^\alpha\,dx \end{equation}
Moreover, the chain of inequalities is sharp since the function
\( u=\varphi\left(\frac{\nu_1(p,\theta)}{R}\rho\right) \in W_0^{1,p}(B_R^\rho, |\nabla_\gamma \rho|^\al \rho^{\theta-Q}) \)
attains both equalities.
\end{theorem}

\begin{theorem}[Hardy Improvement of type I]\label{6Grushin: THM: Hardy Improvement of type I}
Let \(p \geq 2\), \(\alpha \geq 0\), \(\theta \in \mathbb{R}\), and \(\lambda_p\) be the constant defined in Theorem \ref{4: THM: Stime dal basso}. For every \(u \in C^\infty_c(B^\rho_R \setminus \{0\})\), we have:

\[
\int_{B^\rho_R} \frac{|\nabla_\gamma u|^p}{\rho^{p(\theta - 1)}}|\nabla_\gamma \rho|^\alpha \, dx \geq \int_{B^\rho_R} \left|\nabla_\gamma u \cdot \frac{\nabla_\gamma \rho}{|\nabla_\gamma \rho|}\right|^p \frac{|\nabla_\gamma \rho|^\alpha}{\rho^{p(\theta-1)}} \, dx
\]
\begin{equation}\label{6Grushin: EQ: THM: Hardy Improvement of Type I}
\geq \left|\frac{Q - p\theta}{p}\right|^p \int_{B^\rho_R} \frac{|u|^p}{\rho^{p\theta}}|\nabla_\gamma \rho|^{\alpha + p} \, dx + \lambda_p \frac{\nu_1(p,\theta)^p}{R^p} \int_{B^\rho_R} \frac{|u|^p}{\rho^{p(\theta - 1)}}|\nabla_\gamma \rho |^{\alpha + p} \, dx.
\end{equation}
\end{theorem}


\begin{theorem}[Hardy Improvement of Type II]\label{6Grushin: THM: Hardy Improvement of Type II}
Let \( p \geq 2 \), \(\al\geq 0\), \( \theta \in \mathbb{R} \), and \( z_0 \) be the first zero of the Bessel function \( J_0(r) \). Then, for every \( u \in C^\infty_c(B_R^\rho \setminus \{0\}) \), we have:

\[
\int_{B^\rho_R} \frac{|\nabla_\gamma u|^p}{\rho^{p(\theta - 1)}}|\nabla_\gamma \rho|^\alpha \, dx \geq \int_{B^\rho_R} \left|\nabla_\gamma u \cdot \frac{\nabla_\gamma \rho}{|\nabla_\gamma \rho|}\right|^p \frac{|\nabla_\gamma \rho|^\alpha}{\rho^{p(\theta-1)}} \, dx
\]
\begin{equation}\label{6Grushin: EQ: THM: Hardy improvement of type II}
\geq \left|\frac{Q - p\theta}{p}\right|^p \int_{B^\rho_R} \frac{|u|^p}{\rho^{p\theta}}|\nabla_\gamma \rho|^{\alpha + p} \, dx + \frac{2}{p} \left| \frac{Q - p\theta}{p} \right|^{p-2} \frac{z_0^2}{R^2} \int_{B^\rho_R} \frac{|u|^p}{\rho^{p\theta - 2}}|\nabla_\gamma \rho|^{\alpha + p} \, dx
\end{equation}
\end{theorem}

\subsection{Heisenberg Sub-Laplacian}

Let \( N = 2n + 1 \) and let \( (z,t) = (x,y,t) \in \mathbb{R}^n \times \mathbb{R}^n \times \mathbb{R} = \mathbb{H}_n \). We define the vector field \(\nabla_\mathbb{H} = (X_1, \ldots, X_n, Y_1, \ldots, Y_n) \), where 
\[ X_i = \frac{\partial}{\partial x_i} + 2y_i \frac{\partial}{\partial t} \quad \text{and} \quad Y_i = \frac{\partial}{\partial y_i} - 2x_i \frac{\partial}{\partial t}. \]
Note that this is equivalent to choosing the matrix \(\sigma\) equal to 
\[ \begin{pmatrix} I_n & 0 & 2y\\ 0 & I_n & -2x \end{pmatrix}. \]

The corresponding Heisenberg sub-Laplacian is given by \[\Delta_\mathbb{H}u = 
\Delta_z u+4|z|^2 \frac{\partial^2}{\partial t^2} u +4\frac{\partial}{\partial t}(Tu), \quad \text{where } T=\sum_{i=1}^n y_i\frac{\partial}{\partial x_i}-x_i\frac{\partial}{\partial y_i}.\]
There exists a natural family of dilations on \(\mathbb{H}_n\), \(\delta_\lambda : \mathbb{R}^{2n+1}\to \mathbb{R}^{2n+1}\), \(\delta_\lambda(x,y,t)=(\lambda x , \lambda y, \lambda^2 t)\), with respect to which the vector field \(\nabla_\mathbb{H}\) is homogeneous of degree one.
Let \(Q = 2n + 2\) be the homogeneous dimension, and \(\rho(z,t) = (|z|^4 + t^2)^{\frac{1}{4}}\) denote the Koranyi norm. It's straightforward to verify that \(\rho\) is homogeneous of degree one with respect to \(\delta_\lambda\).
Furthermore, considering the Heisenberg \(p\)-sub-Laplacian, \(\Delta_{p,\mathbb{H}} = \text{div}_\mathbb{H}(|\nabla_\mathbb{H} u|^{p-2}\nabla_\mathbb{H} u)\), it is well known that (see \cite{BT})
\[
\begin{cases} 
\Delta_{p,\mathbb{H}} \rho^{\frac{p-Q}{p-1}} = 0 \text{ in } \mathbb{H}_n & \text{if } p \neq Q \\
\Delta_{Q,\mathbb{H}} (-\ln \rho) = 0 \text{ in } \mathbb{H}_n & \text{if } p = Q.
\end{cases}
\]

For the sake of completeness, we also observe that \(|\nabla_\mathbb{H}\rho|^2 = \frac{|z|^2}{\rho^2}\).
We can apply all the results from Section \ref{SEC: Main Results} to this context.

\begin{theorem}[Weighted Poincaré inequality]\label{6Heisenebrg: THM: Disuguaglianza di Poicnarè pesata}
Let \( p \geq 2 \), \(\al\geq 0\) and \( \theta \geq 1 \) be fixed. For every \(u\in W_0^{1,p}(B^\rho_R, |\nabla_\mathbb{H} \rho|^\al \rho^{\theta-Q})\), the following inequalities hold:
\begin{equation}\label{6heisenebrg: EQ: THM: Weighted Poincaré Inequality}
\left(\frac{\nu_1(p,\theta)}{R}\right)^p\int_{B_R^\rho} \frac{|u|^p}{\rho^{Q-\theta}}|\nabla_\mathbb{H} \rho|^{\alpha+p}\,dx 
\leq \int_{B_R^\rho} \left|\nabla_\mathbb{H} u \cdot\frac{\nabla_\mathbb{H} \rho}{|\nabla_\mathbb{H} \rho|}\right|^p \frac{|\nabla_\mathbb{H} \rho|^\alpha}{\rho^{Q-\theta}}\,dx
\leq \int_{B_R^\rho} \frac{|\nabla_\mathbb{H} u|^p}{\rho^{Q-\theta}}|\nabla_\mathbb{H} \rho |^\alpha\,dx
\end{equation}
Moreover, the chain of inequalities is sharp since the function
\( u=\varphi\left(\frac{\nu_1(p,\theta)}{R}\rho\right) \in W_0^{1,p}(B_R^\rho, |\nabla_\mathbb{H} \rho|^\al \rho^{\theta-Q}) \)
attains both equalities.
\end{theorem}

\begin{theorem}[Hardy Improvement of type I]\label{6Heisenebrg: THM: Hardy Improvement of type I}
Let \(p \geq 2\), \(\alpha \geq 0\), \(\theta \in \mathbb{R}\), and \(\lambda_p\) be the constant defined in Theorem \ref{4: THM: Stime dal basso}. For every \(u \in C^\infty_c(B^\rho_R \setminus \{0\})\), we have:

\[
\int_{B^\rho_R} \frac{|\nabla_\mathbb{H} u|^p}{\rho^{p(\theta - 1)}}|\nabla_\mathbb{H} \rho|^\alpha \, dx \geq \int_{B^\rho_R} \left|\nabla_\mathbb{H} u \cdot \frac{\nabla_\mathbb{H} \rho}{|\nabla_\mathbb{H} \rho|}\right|^p \frac{|\nabla_\mathbb{H} \rho|^\alpha}{\rho^{p(\theta-1)}} \, dx
\]
\begin{equation}\label{6Heisenebrg: EQ: THM: Hardy Improvement of Type I}
\geq \left|\frac{Q - p\theta}{p}\right|^p \int_{B^\rho_R} \frac{|u|^p}{\rho^{p\theta}}|\nabla_\mathbb{H} \rho|^{\alpha + p} \, dx + \lambda_p \frac{\nu_1(p,\theta)^p}{R^p} \int_{B^\rho_R} \frac{|u|^p}{\rho^{p(\theta - 1)}}|\nabla_\mathbb{H} \rho |^{\alpha + p} \, dx.
\end{equation}
\end{theorem}


\begin{theorem}[Hardy Improvement of Type II]\label{6Heisnebrg: THM: Hardy Improvement of Type II}
Let \( p \geq 2 \), \(\al\geq 0\), \( \theta \in \mathbb{R} \), and \( z_0 \) be the first zero of the Bessel function \( J_0(r) \). Then, for every \( u \in C^\infty_c(B_R^\rho \setminus \{0\}) \), we have:

\[
\int_{B^\rho_R} \frac{|\nabla_\mathbb{H} u|^p}{\rho^{p(\theta - 1)}}|\nabla_\mathbb{H} \rho|^\alpha \, dx \geq \int_{B^\rho_R} \left|\nabla_\mathbb{H} u \cdot \frac{\nabla_\mathbb{H} \rho}{|\nabla_\mathbb{H} \rho|}\right|^p \frac{|\nabla_\mathbb{H} \rho|^\alpha}{\rho^{p(\theta-1)}} \, dx
\]
\begin{equation}\label{6Heisneberg: EQ: THM: Hardy improvement of type II}
\geq \left|\frac{Q - p\theta}{p}\right|^p \int_{B^\rho_R} \frac{|u|^p}{\rho^{p\theta}}|\nabla_\mathbb{H} \rho|^{\alpha + p} \, dx + \frac{2}{p} \left| \frac{Q - p\theta}{p} \right|^{p-2} \frac{z_0^2}{R^2} \int_{B^\rho_R} \frac{|u|^p}{\rho^{p\theta - 2}}|\nabla_\mathbb{H} \rho|^{\alpha + p} \, dx
\end{equation}
\end{theorem}

\subsection{The Heisenberg-Greiner operator}
In \(\mathbb{R}^{2n+1}\), let's consider the vector fields:

\begin{align*}
X_i &= \frac{\partial}{\partial x_i} + 2\gamma y_i |z|^{2\gamma - 2} \frac{\partial}{\partial t}\quad i=1,\ldots,n \\
Y_i &= \frac{\partial}{\partial y_i} - 2\gamma x_i |z|^{2\gamma - 2} \frac{\partial}{\partial t}
\end{align*}

Here, \(\gamma \geq 1\) is a fixed parameter, and \((z,t) = (x,y,t)\) represents a generic element of \(\mathbb{R}^{2n} \times \mathbb{R}\). It's noteworthy that for \(p=2\) and \(\gamma=1\), the operator \(\mathcal{L}_p\) corresponds to the Heisenberg sub-Laplacian \(\Delta_\mathbb{H}\). When \(p=2\) and \(\gamma > 1\), \(\mathcal{L}_p\) is a Greiner operator, \cite{G}.

Let us define the function \(\rho(z,t) = (|z|^{4\gamma} + t^2)^{\frac{1}{4\gamma}}\) and observe that both \(\rho\) and the vector fields \(X_i, Y_i\) are homogeneous of degree one with respect to the dilation family defined as follows:

\[
\delta_\lambda : \mathbb{R}^{2n+1} \to \mathbb{R}^{2n+1}, \quad \delta_\lambda(x,y,t) = (\lambda x, \lambda y, \lambda^{2\gamma} t).
\]

Let \(Q = 2n + 2\gamma\) denote the homogeneous dimension, and let \(\Gamma_p(z,t)\) be defined as:

\[
\Gamma_p(z,t) = \begin{cases} 
\rho^{\frac{p-Q}{p-1}} & \text{se } p \neq Q \\
-\ln \rho & \text{se } p = Q.
\end{cases}
\]

It is a well-known fact that \(\Gamma_p\) is \(\mathcal{L}_p\)-harmonic in \(\mathbb{R}^{2n+1}\) (see \cite{DA0} and \cite{ZN}). Moreover, a simple calculation shows that \(|\nabla_{\mathcal{L}} \rho| = \frac{|z|^{2\gamma-1}}{\rho^{2\gamma - 1}}\). We can now apply the results of Section \ref{SEC: Main Results}.

\begin{theorem}[Weighted Poincaré inequality]\label{6HeisenebrgGreiner: THM: Disuguaglianza di Poicnarè pesata}
Let \( p \geq 2 \), \(\al\geq 0\) and \( \theta \geq 1 \) be fixed. For every \( u \in W_0^{1,p}(B^\rho_R, |\nabla_\L \rho|^\al \rho^{\theta-Q}) \), the following inequalities hold:
\begin{equation}\label{6heisenebrgGreiner: EQ: THM: Weighted Poincaré Inequality}
\left(\frac{\nu_1(p,\theta)}{R}\right)^p\int_{B_R^\rho} \frac{|u|^p}{\rho^{Q-\theta}}|\nabla_\mathcal{L} \rho|^{\alpha+p}\,dx 
\leq \int_{B_R^\rho} \left|\nabla_\mathcal{L} u \cdot\frac{\nabla_\mathcal{L} \rho}{|\nabla_\mathcal{L} \rho|}\right|^p \frac{|\nabla_\mathcal{L} \rho|^\alpha}{\rho^{Q-\theta}}\,dx
\leq \int_{B_R^\rho} \frac{|\nabla_\mathcal{L} u|^p}{\rho^{Q-\theta}}|\nabla_\mathcal{L} \rho |^\alpha\,dx
\end{equation}
Moreover, the chain of inequalities is sharp since the function
\( u=\varphi\left(\frac{\nu_1(p,\theta)}{R}\rho\right) \in W_0^{1,p}(B_R^\rho, |\nabla_\L \rho|^\al \rho^{\theta-Q}) \)
attains both equalities.
\end{theorem}

\begin{theorem}[Hardy Improvement of type I]\label{6HeisenebrgGreiner: THM: Hardy Improvement of type I}
Let \(p \geq 2\), \(\alpha \geq 0\), \(\theta \in \mathbb{R}\), and \(\lambda_p\) be the constant defined in Theorem \ref{4: THM: Stime dal basso}. For every \(u \in C^\infty_c(B^\rho_R \setminus \{0\})\), we have:

\[
\int_{B^\rho_R} \frac{|\nabla_\mathcal{L} u|^p}{\rho^{p(\theta - 1)}}|\nabla_\mathcal{L} \rho|^\alpha \, dx \geq \int_{B^\rho_R} \left|\nabla_\mathcal{L} u \cdot \frac{\nabla_\mathcal{L} \rho}{|\nabla_\mathcal{L} \rho|}\right|^p \frac{|\nabla_\mathcal{L} \rho|^\alpha}{\rho^{p(\theta-1)}} \, dx
\]
\begin{equation}\label{6HeisenebrgGreiner: EQ: THM: Hardy Improvement of Type I}
\geq \left|\frac{Q - p\theta}{p}\right|^p \int_{B^\rho_R} \frac{|u|^p}{\rho^{p\theta}}|\nabla_\mathcal{L} \rho|^{\alpha + p} \, dx + \lambda_p \frac{\nu_1(p,\theta)^p}{R^p} \int_{B^\rho_R} \frac{|u|^p}{\rho^{p(\theta - 1)}}|\nabla_\mathcal{L} \rho |^{\alpha + p} \, dx.
\end{equation}
\end{theorem}


\begin{theorem}[Hardy Improvement of Type II]\label{6HeisnebrgGreiner: THM: Hardy Improvement of Type II}
Let \( p \geq 2 \), \(\al\geq 0\), \( \theta \in \mathbb{R} \), and \( z_0 \) be the first zero of the Bessel function \( J_0(r) \). Then, for every \( u \in C^\infty_c(B_R^\rho \setminus \{0\}) \), we have:

\[
\int_{B^\rho_R} \frac{|\nabla_\mathcal{L} u|^p}{\rho^{p(\theta - 1)}}|\nabla_\mathcal{L} \rho|^\alpha \, dx \geq \int_{B^\rho_R} \left|\nabla_\mathcal{L} u \cdot \frac{\nabla_\mathcal{L} \rho}{|\nabla_\mathcal{L} \rho|}\right|^p \frac{|\nabla_\mathcal{L} \rho|^\alpha}{\rho^{p(\theta-1)}} \, dx
\]
\begin{equation}\label{6HeisnebergGreiner: EQ: THM: Hardy improvement of type II}
\geq \left|\frac{Q - p\theta}{p}\right|^p \int_{B^\rho_R} \frac{|u|^p}{\rho^{p\theta}}|\nabla_\mathcal{L} \rho|^{\alpha + p} \, dx + \frac{2}{p} \left| \frac{Q - p\theta}{p} \right|^{p-2} \frac{z_0^2}{R^2} \int_{B^\rho_R} \frac{|u|^p}{\rho^{p\theta - 2}}|\nabla_\mathcal{L} \rho|^{\alpha + p} \, dx
\end{equation}
\end{theorem}

\subsection{Sub-Laplacian on Homogeneous Carnot group}
We begin by quoting some preliminary facts about the homogeneous Carnot group and refer the interested reader to \cite{BLU} for more detailed information on this subject.

\begin{definition}
 We say that a Lie group on \(\rN\), 
\(G=(\rN,\circ)\), is a (homogeneous) Carnot group  if the following properties hold: 
\begin{enumerate}
\item \(\rN\) can be split as \(\rN=\erre^{N_1}\times\ldots\times\erre^{N_r}\),  and the dilation \(\delta_\la\colon \rN\to\rN\)
\[
\delta_\la(x)=\delta_\la(x^{(1)},\ldots,x^{(r)})
=
(\lambda x^{(1)},\lambda^2 x^{(2)},\ldots,\lambda^r x^{(r)})
\quad x^{(i)}\in\erre^{N_i},
\]
is an automorphism of the group \(G\) for every \(\la>0\).
\item  If \(N_1\) is as above, let \(X_1,\ldots,X_{N_1}\) 
be the left invariant vector fields on \(G\) such
that \(\dyle X_j(0)\tl=\tl\frac{\pa}{\pa x_j}\biggl|_0\) for \(j=1,\ldots, N_1.\) Then
\[
\mathrm{rank}\left(\mathrm{Lie}\{X_1,\ldots,X_{N_1}\}(x)\right)=N\quad\text{  for every } x\in\rN.
\]
\end{enumerate}
\end{definition}

We denote by \(\nabla_\mathcal{L}\) the vector field \(\nabla_\mathcal{L} = (X_1, \ldots, X_{N_1})\). The canonical sub-Laplacian on \(G\) is the second-order differential operator defined by \(\mathcal{L}_2 = \sum_{i=1}^{N_1} X_i^2\), and for \(p \geq 2\), the \(p\)-sub-Laplacian operator is given by \(\mathcal{L}_p u = \sum_{i=1}^{N_1} X_i(|\nabla_\mathcal{L} u|^{p-2} X_i u) = \text{div}_\mathcal{L} (|\nabla_\mathcal{L} u|^{p-2} \nabla_\mathcal{L} u)\).

Some important properties of homogeneous Carnot groups are as follows: The Lebesgue measure on \(\mathbb{R}^N\) coincides with the bi-invariant Haar measure on \(G\). It is easy to verify that, denoting by \(Q = N_1 + 2N_2 + \ldots + rN_r\) the homogeneous dimension of \(G\), for any measurable subset \(E \subseteq \mathbb{R}^N\), we have \(|\delta_\lambda(E)| = \lambda^Q|E|\), here \(|E|\) represents the Lebesgue measure of \(E\). Furthermore, the vector fields \(X_1, \ldots, X_{N_1}\) are homogeneous of degree one with respect to \(\delta_\lambda\).\\

Special examples of Carnot groups are the Euclidean spaces \(\mathbb{R}^N\). A non-trivial example of a Carnot group is the Heisenberg group. If \(Q < 3\), then any Carnot group is the ordinary Euclidean space \(\mathbb{R}^Q\). From now on, we can assume \(Q \geq 3\).\\

We call a homogeneous norm on \(G\) every continuous function \(\rho: G \to [0, +\infty)\) such that 
\begin{align*}
&\rho(\delta_\lambda(x)) = \lambda \rho(x) \quad \text{for every} \ \lambda > 0 \ \text{and} \ x \in G \\
&\rho(x) > 0 \quad \text{if and only if} \ x \neq 0.
\end{align*}

While it is a well-known fact that there always exists a homogeneous norm \(\rho_2 \in C^\infty(G\setminus\{0\})\) such that \(\Gamma(x) = \rho_2^{2-Q}(x)\) is the fundamental solution for \(-\mathcal{L}_2\), it is not guaranteed that the same \(\rho_2\) is also a fundamental solution for \(\mathcal{L}_p\). We give the following definition (see \cite{DA0}).

\begin{definition}
We say that a Carnot group G is \textit{polarizable} if 
\[
\begin{cases} 
\L_p \rho_2^{\frac{p-Q}{p-1}} = 0 \text{ in } G\setminus\{0\} & \text{if } p \neq Q \\
\L_Q(-\ln \rho_2) = 0 \text{ in } G\setminus\{0\} & \text{if } p = Q.
\end{cases}
\]
\end{definition}

Examples of polarizable Carnot groups include the usual Euclidean spaces, the Heisenberg group, and H-type groups. This is proved in \cite{BT}.

\begin{theorem}[Weighted Poincaré inequality]\label{6Carnot: THM: Disuguaglianza di Poicnarè pesata}
Let \(G\) be a polarizable Carnot group and
let \( p \geq 2 \), \(\al\geq 0\) and \( \theta \geq 1 \) be fixed. For every \( u \in W_0^{1,p}(B^{\rho_2}_R, |\nabla_\L \rho_2|^\al \rho_2^{\theta-Q}) \), the following inequalities hold:
\begin{equation}\label{6Carnot: EQ: THM: Weighted Poincaré Inequality}
\left(\frac{\nu_1(p,\theta)}{R}\right)^p\int_{B_R^{\rho_2}} \frac{|u|^p}{\rho_2^{Q-\theta}}|\nabla_\mathcal{L} \rho_2|^{\alpha+p}\,dx 
\leq \int_{B_R^{\rho_2}} \left|\nabla_\mathcal{L} u \cdot\frac{\nabla_\mathcal{L} \rho_2}{|\nabla_\mathcal{L} \rho_2|}\right|^p \frac{|\nabla_\mathcal{L} \rho_2|^\alpha}{\rho_2^{Q-\theta}}\,dx
\leq \int_{B_R^{\rho_2}} \frac{|\nabla_\mathcal{L} u|^p}{\rho_2^{Q-\theta}}|\nabla_\mathcal{L} \rho_2 |^\alpha\,dx
\end{equation}
Moreover, the chain of inequalities is sharp since the function
\( u=\varphi\left(\frac{\nu_1(p,\theta)}{R}\rho_2\right) \in W_0^{1,p}(B_R^{\rho_2}, |\nabla_\L \rho_2|^\al \rho_2^{\theta-Q}) \)
attains both equalities.
\end{theorem}

\begin{theorem}[Hardy Improvement of type I]\label{6Carnot: THM: Hardy Improvement of type I}
Let \(G\) be a polarizable Carnot group.
Let \(p \geq 2\), \(\alpha \geq 0\), \(\theta \in \mathbb{R}\), and \(\lambda_p\) be the constant defined in Theorem \ref{4: THM: Stime dal basso}. For every \(u \in C^\infty_c(B^{\rho_2}_R \setminus \{0\})\), we have:

\[
\int_{B^{\rho_2}_R} \frac{|\nabla_\mathcal{L} u|^p}{\rho_2^{p(\theta - 1)}}|\nabla_\mathcal{L} \rho_2|^\alpha \, dx \geq \int_{B^{\rho_2}_R} \left|\nabla_\mathcal{L} u \cdot \frac{\nabla_\mathcal{L} \rho_2}{|\nabla_\mathcal{L} \rho_2|}\right|^p \frac{|\nabla_\mathcal{L} \rho_2|^\alpha}{\rho_2^{p(\theta-1)}} \, dx
\]
\begin{equation}\label{6Carnot: EQ: THM: Hardy Improvement of Type I}
\geq \left|\frac{Q - p\theta}{p}\right|^p \int_{B^{\rho_2}_R} \frac{|u|^p}{\rho_2^{p\theta}}|\nabla_\mathcal{L} \rho_2|^{\alpha + p} \, dx + \frac{2}{p} \left| \frac{Q - p\theta}{p} \right|^{p-2} \frac{z_0^2}{R^2} \int_{B^{\rho_2}_R} \frac{|u|^p}{\rho_2^{p\theta - 2}}|\nabla_\mathcal{L} \rho_2|^{\alpha + p} \, dx
\end{equation}
\end{theorem}


\begin{theorem}[Hardy Improvement of Type II]\label{6Carnot: THM: Hardy Improvement of Type II}
Let \(G\) be a polarizable Carnot group and
let \( p \geq 2 \), \( \alpha \geq 0\), \( \theta \in \mathbb{R} \), and \( z_0 \) be the first zero of the Bessel function \( J_0(r) \). Then, for every \( u \in C^\infty_c(B_R^{\rho_2} \setminus \{0\}) \), we have:

\[
\int_{B^{\rho_2}_R} \frac{|\nabla_\mathcal{L} u|^p}{\rho_2^{p(\theta - 1)}}|\nabla_\mathcal{L} \rho_2|^\alpha \, dx \geq \int_{B^{\rho_2}_R} \left|\nabla_\mathcal{L} u \cdot \frac{\nabla_\mathcal{L} \rho_2}{|\nabla_\mathcal{L} \rho_2|}\right|^p \frac{|\nabla_\mathcal{L} \rho_2|^\alpha}{\rho_2^{p(\theta-1)}} \, dx
\]
\begin{equation}\label{6Carnot: EQ: THM: Hardy improvement of type II}
\geq \left|\frac{Q - p\theta}{p}\right|^p \int_{B^{\rho_2}_R} \frac{|u|^p}{\rho_2^{p\theta}}|\nabla_\mathcal{L} \rho_2|^{\alpha + p} \, dx + \frac{2}{p} \left| \frac{Q - p\theta}{p} \right|^{p-2} \frac{z_0^2}{R^2} \int_{B^{\rho_2}_R} \frac{|u|^p}{\rho_2^{p\theta - 2}}|\nabla_\mathcal{L} \rho_2|^{\alpha + p} \, dx.
\end{equation}
\end{theorem}

\vspace{0.5cm}

\appendix

\section{Appendix}\label{Appendice A}
In this appendix, we provide a detailed proof of Proposition \ref{1: PROP: Esistenza Unicità e Oscillazioni} and Proposition \ref{1: THM: Teorema Fondamentale}.\\\\
From now on, we fix \(p \geq 2\). We recall the notation \(\phi_p(x) = |x|^{p-2}x\) for \(x \in \mathbb{R}\). 
%
It's easily verified that \(\phi_p(x)\in C^1(\erre)\), \(|\phi_p(x)| = \phi_p(|x|)\) and \(\frac{\pa}{\pa x}\phi_p(x)=(p-1)|x|^{p-2}\). 
If \(\phi_p(x)=y\), then \(x=\phi_{p'}(y)\) where \(\frac{1}{p}+\frac{1}{p'}=1\). Note that since \(p\geq 2\), \(\phi_{p'}(x)\) is not differentiable at \(x=0\), but \(\phi_{p'}(x) \in C^1(\mathbb{R}\diff0)\). Finally, both \(\phi_p(x)\) and \(\phi_{p'}(x)\) are increasing functions.

%
%
%
%
%
%

Here is the first Lemma.
\begin{lemma}\label{1: Lemma: Regolarità in 0}
Let \(\theta > 1\) and suppose we have a function \(h \in C^1[0,a]\), with \(a > 0\), such that \(r^{\theta-1}\phi_p(h') \in C^1(0,a)\) and \(h\) solves:
\begin{equation}\label{1: EQ: Lemma: Regolarità in 0}
(r^{\theta-1}\phi_p(h'))' + \la r^{\theta-1}\phi_p(h) = 0 \,\,\,\text{in } (0,a), \text{ for some \(\lambda > 0\)},
\end{equation}
then 
\begin{enumerate}
\item \(h'(0)=0\),

\item \(r^{\theta-1}\phi_p(h')\in C^1[0,a)\),

\item \((r^{\theta-1}\phi_p(h'))'(0)=0\).
\end{enumerate}
\end{lemma}

\begin{remark}\label{1: Osservazione: Il caso theta=1}
The conclusions 1. and 3. of lemma \ref{1: Lemma: Regolarità in 0} are no longer true if \(\theta = 1\). For example, if \(p=2\), the function \(h(r) = \sin(r)\) solves the equation \(h''(r) + h(r) = 0\) in a neighborhood of the origin (\(\la=1\)), but \(h'(0) = 1\).
In any case, conclusion 2. remains true. If \(h \in C^1[0,a]\) and \(|h'|^{p-2}h'\in C^1(0,a)\) 
then, by integrating the equation from 0 to \(r\) and dividing by \(r\), we obtain
\[
(|h'|^{p-2}h')'(0) := \lim_{r \to 0^+} \frac{|h'(r)|^{p-2}h'(r) - |h'(0)|^{p-2}h'(0)}{r} = \lim_{r \to 0^+} \frac{\la}{r} \int_0^r |h|^{p-2}h(s)\,ds =\la |h(0)|^{p-2}h(0).
\]
\end{remark}

\begin{proof}
Integrating equation (\ref{1: EQ: Lemma: Regolarità in 0}) from 0 to \(r\), we have:
\[
r^{\theta-1}\phi_p(h')(r) + \lambda \int_0^r s^{\theta-1}\phi_p(h)(s)\,ds = 0.
\]
Applying the change of variable \(t=s/r\) to the integral, we arrive at the expression:
\[
\frac{|h'|^{p-2}h'(r)}{r} = -\lambda \int_0^1 t^{\theta-1}\phi_p(h)(rt)\,dt.
\]
Since \(h \in C^1 [0,a]\), the right-hand side remains bounded, therefore,
\(\dyle
\frac{|h'|^{p-2}h'(r)}{r} \quad \text{remains bounded as } r \to 0^+.
\)
From this, we can conclude \(h'(0) = 0\).
For points 2. and 3.,  it is sufficient to observe that
\[
(r^{\theta-1}\phi_p(h'))'(0) := \lim_{r \to 0^+} \frac{r^{\theta-1}\phi_p(h')(r)}{r} = \lim_{r \to 0^+} -\frac{\lambda}{r} \int_0^r s^{\theta-1}\phi_p(h)(s)\,ds = 0
.
\]
\end{proof}

Let's now proceed to study the nonlinear differential equation \((r^{\theta-1}\phi_p(h'))' + r^{\theta-1}\phi_p(h) = 0 \quad \text{for } r \geq 0.\)
\begin{theorem}[Existence and uniqueness]\label{1: THM: Esistenza e Unicità}
For fixed \(r_0 \geq 0\) and \(\theta \geq 1\), there exists a unique solution \(\varphi(r) \in C^1[r_0, +\infty)\), with \(r^{\theta-1}\phi_p(\varphi') \in C^1[r_0, +\infty)\), to the problem:
\begin{equation}\label{1: EQ: THM: Esistenza e Unicità}
\quad
\begin{cases} 
(r^{\theta-1}\phi_p(h'))' + r^{\theta-1}\phi_p(h) = 0  & \text{for } r \geq r_0\\ 
h(r_0) = h_0, \quad h'(r_0) = h'_0 
\end{cases}
\end{equation}
where \(h'_0 = 0\) if \(r_0 = 0\).
\end{theorem}

\begin{remark}
As stated in lemma \ref{1: Lemma: Regolarità in 0}, when \(\theta > 1\), the condition \(h'_0 = 0\) for \(r_0 = 0\) is necessary to ensure the existence of a solution. However, for \(\theta = 1\), we impose this condition.
\end{remark}

\begin{proof}
We split the proof into four steps.\\\\
\textit{Step 1) Local Existence.}\\
We prove that there exists a local solution in a right interval of the form \([r_0, r_0 + \delta_1]\), with \(\delta_1\) to be chosen later, depending only on \(h_0\) and \(h'_0\), and independent of \(r_0\).\\

Integrating equation (\ref{1: EQ: THM: Esistenza e Unicità}) from \(r_0\) to \(r\), we obtain:
\begin{equation}\label{1: EQ: DIM: THM: Locale Esistenza: Prima relazione integrale}
r^{\theta-1}\phi_p(h')(r)-r_0^{\theta-1}\phi_p(h'_0)+\int_{r_0}^r s^{\theta-1}\phi_p(h)(s)\,ds=0,\quad\text{in other words,}
\end{equation}
\[
\phi_p(h')(r) = \left(\frac{r_0}{r}\right)^{\theta-1} \phi_p(h'_0) - \int_{r_0}^{r} \left(\frac{s}{r}\right)^{\theta-1} \phi_p(h)(s)\,ds.
\]
From this, we can express \(h'(r)\) as
\[
h'(r) = \phi_{p'}\left( \left(\frac{r_0}{r}\right)^{\theta-1} \phi_p(h'_0) - \int_{r_0}^{r} \left(\frac{s}{r}\right)^{\theta-1} \phi_p(h)(s)\,ds \right),
\]
and, with another integration from \(r_0\) to \(r\), we have:
\[ 
h(r) = h_0 + \int_{r_0}^{r} \phi_{p'}\left( \left(\frac{r_0}{t}\right)^{\theta-1} \phi_p(h'_0) - \int_{r_0}^{t} \left(\frac{s}{t}\right)^{\theta-1} \phi_p(h)(s)\,ds \right)dt.
\]

We define the operator \(T\) as follows:
\[ 
T(h)(r) := h_0 + \int_{r_0}^{r} \phi_{p'}\left( \left(\frac{r_0}{t}\right)^{\theta-1} \phi_p(h'_0) - \int_{r_0}^{t} \left(\frac{s}{t}\right)^{\theta-1} \phi_p(h)(s)\,ds \right)dt.
\]
We can summarize by saying that \(h\) is a solution to the differential equation if and only if \(h\) is a fixed point of the operator \(T\).
Our strategy is the following: consider \(X = C[r_0,r_0+\delta_1]\) as a Banach space, and note that \(T: X \to X\). We define \(S\) the closed, bounded, convex, and non-empty subset of \(X\) as \(S = \overline{B_R(0)}\subset C[r_0, r_0 + \delta_1]\). Our goal is to select \(\delta_1\) and \(R>0\) such that \(T: S \to S\) becomes a compact operator. By applying the Schauder fixed-point theorem, we can then conclude that \(T\) has at least one fixed point \(h \in S\).\\

Given \(h \in C[r_0,r_0+\delta_1]\), we have:
\[
\begin{split}
|T(h)(r)|
&
\leq |h_0| + \int_{r_0}^{r} \phi_{p'}\left( \left| \left(\frac{r_0}{t}\right)^{\theta-1} \phi_p(h'_0) - \int_{r_0}^{t} \left(\frac{s}{t}\right)^{\theta-1} \phi_p(h)(s)\,ds \right| \right)\,dt  \\
& 
\leq |h_0| + \int_{r_0}^{r} \phi_{p'}\left( |h'_0|^{p-1} + \int_{r_0}^{t} \left(\frac{s}{t}\right)^{\theta-1} |h|^{p-1}(s)\,ds \right)\,dt  \\
& 
\leq |h_0| + \int_{r_0}^{r} \phi_{p'}\left( |h'_0|^{p-1}+\left( \sup_{[r_0,r_0+\delta_1]} |h|\right)^{p-1} \frac{t^\theta-r_0^\theta}{\theta t^{\theta-1}}\right)\,dt \\
& 
= |h_0| + \int_{r_0}^{r} \left(|h'_0|^{p-1}+\left( \sup_{[r_0,r_0+\delta_1]} |h|\right)^{p-1} \frac{t^\theta-r_0^\theta}{\theta t^{\theta-1}}\right)^{\frac{1}{p-1}}dt
\end{split}
\]
For \(\theta \geq 1\), it is easy to verify that \(\dyle \frac{t^\theta-r_0^\theta}{\theta t^{\theta-1}} \leq t-r_0\) using Lagrange theorem. Moreover, for \(\alpha \in (0,1]\), and thus in particular for \(\alpha = \frac{1}{p-1}\), the triangle inequality \(|x+y|^\alpha \leq |x|^\alpha+|y|^\alpha\) holds. Using these two observations, we can write:
\[
\begin{split}
|T(h)(r)| & \leq |h_0|+\int_{r_0}^{r} \left(|h'_0|+\delta_1^{\frac{1}{p-1}} \sup_{[r_0, r_0+\delta_1]}|h|\right) 
\leq |h_0|+|h'_0|+\delta_1^{p'} \sup_{[r_0, r_0+\delta_1]}|h|.
\end{split}
\]
Denoting with \(\|h\|\) the norm in \(C[r_0, r_0+\delta_1]\), we have found
\[
\|T(h)\| \leq C + \delta_1^{p'} \|h\| \quad \text{where } C = |h_0| + |h'_0|.
\]
Now let's consider \(R =2 C\) and \(S = \overline{B_R(0)} \subset C[r_0, r_0+\delta_1]\). It is easy to verify that
\begin{equation}\label{1: EQ: DIM: THM: Locale Esistenza: scelta delta}
\text{if }\,\delta_1^{p'} \leq \frac{R-C}{R}=\frac{1}{2}\,\text{ then }\,T: S \to S.
\end{equation}
From Ascoli-Arzelà theorem, we can conclude that \(T|_S\) is a compact operator. 
Indeed, let's suppose we have a bounded sequence \(h_n \in C[r_0, r_0+\delta_1]\) with \(\|h_n\| \leq M\). Then:

- \textit{Equiboundedness:} \(\|T(h_n)\| \leq C + \delta_1^{p'} M \leq M'\).
  
- \textit{Equicontinuity:} With calculations similar to those done earlier, we can easily verify that the sequence \(T(h_n)\) is equicontinuous.


From the Schauder fixed-point theorem, we can conclude that there exists a function \(h \in C[r_0, r_0+\delta_1]\) such that \(h(r) = T(h)(r)\). In particular, from this identity, we also have \(h \in C^1[r_0, r_0+\delta_1]\), and going back to equation (\ref{1: EQ: DIM: THM: Locale Esistenza: Prima relazione integrale}) 
we can conclude that \(r^{\theta-1}\phi_p(h') \in C^1[r_0, r_0+\delta_1]\). This is the local solution we were looking for.\\

\textit{Step 2) Extension of the solution, global existence.}\\
Let \(h\) be the solution in \([r_0, r_0+\delta_1]\) provided in the previous step, and consider the new problem 
\[
\begin{cases} (r^{\theta-1}\phi_p(u'))'+r^{\theta-1}\phi_p(u)=0 \text{ in } [r_1, r_1+\delta_2]\\ u(r_1)=h_1, \quad u'(r_1)=h'_1
\end{cases}
\]
where \(r_1=r_0+\delta_1\), \(h_1=h(r_0+\delta_1)\), and \(h'_1=h'(r_0+\delta_1)\).\\

Using the same arguments as in Step 1, we find a local solution, which we denote as \(\overline{h}\), in the interval \([r_1, r_1 + \delta_2]\), where \(\delta_2\) satisfies 
\[ 
\delta_2^{p'} \leq \frac{\tilde{R} - |h_1| - |h'_1|}{\tilde{R}} =\frac{1}{2}, \quad\text{ if we choose } \tilde{R}=2(|h_1| + |h'_1|).
\]
Therefore, we can assume \(\delta_2 = \delta_1\). Consequently, 
the function
\[
\tilde{h}(r) =
\begin{cases} h(r) & \text{for } r \in [r_0, r_0+\delta_1], \\
\overline{h}(r) & \text{for } r \in [r_0+\delta_1, r_0+2\delta_1],
\end{cases}
\]
is a \(C^1\) function, with \(r^{\theta-1}\phi_p(\tilde{h}') \in C^1\), solving our problem over the entire interval \([r_0, r_0+2\delta_1]\).\\

Iterating this process and advancing by a step of \(\delta_1\) each time, we can find a solution \(\varphi(r) \in C^1[r_0,+\infty)\) with \(r^{\theta-1}\phi_p(\varphi') \in C^1[r_0,+\infty)\).\\

\textit{Step 3) Local uniqueness.}\\
Suppose that \( h_1 \) and \( h_2 \) are two local solutions with the same initial data at a point \( \overline{r} \geq 0 \), meaning \( h_1 \) and \( h_2 \) solve
\[ 
\begin{cases} 
(r^{\theta-1}\phi_p(h'_i))' + r^{\theta-1}\phi_p(h_i) = 0 & \text{in } [\overline{r}, \overline{r} + \delta] \,\,\,\,\text{ for } i = 1, 2 \\
 h_i(\overline{r}) = h_0, \quad h'_i(\overline{r}) = h'_0 
& \text{\( h'_0 = 0 \) if \( \overline{r} = 0 \),}
\end{cases} 
\]
we want to prove that \( h_1 \equiv h_2 \) in \([\overline{r}, \overline{r} + \delta]\).\\

Let's assume, for the sake of contradiction, that \( h_1 \neq h_2 \) and proceed to analyze several different situations.\\

\textit{Case A) \( \overline{r} > 0 \).}\\
The following relation holds
\[
h_i(r) = h_0 + \int_{\overline{r}}^r \phi_{p'}\left(\left(\frac{\overline{r}}{t}\right)^{\theta-1}\phi_p(h'_0) - \int_{\overline{r}}^t \left(\frac{s}{t}\right)^{\theta-1} \phi_p(h_i)(s)\, ds\right) dt.
\]
Let \(\dyle g_i(t) = \left(\frac{\overline{r}}{t}\right)^{\theta-1}\phi_p(h'_0) - \int_{\overline{r}}^t \left(\frac{s}{t}\right)^{\theta-1} \phi_p(h_i)(s)\, ds \). It is easy to verify that \(\dyle \lim_{{t \to \overline{r}}^+} g_i(t) = \phi_p(h'_0) \), indeed:
\[
\left| \int_{\overline{r}}^t \left(\frac{s}{t}\right)^{\theta-1} \phi_p(h_i(s))\,ds \right| \leq \int_{\overline{r}}^t \left(\frac{s}{t}\right)^{\theta-1}|h_i|^{p-1}\,ds \leq \left(\sup_{[\overline{r}, \overline{r}+\delta]} |h_i|\right)^{p-1} \frac{t^{\theta}-\overline{r}^{\theta}}{\theta t^{\theta-1}} \longrightarrow 0 \quad\text{as }\, t\to\overline{r}.  
\]

\textit{Subcase A.1) \( h'_0 \neq 0. \)}\\
In this case, where \( \phi_p(h'_0) \neq 0 \), the continuity of \( g_i \) alongside the assumption of a sufficiently small \( \delta \) guarantee that \( g_i(r) \) maintains non-zero values within the interval \([\overline{r},\overline{r}+\delta]\). Consequently, we can utilize the \( C^1 \) smoothness of \( \phi_{p'} \) and its local Lipschitz property,
\begin{equation}\label{1: EQ: DIM: THM: Locale Unicità: Lip di phi_p'}
|h_1 - h_2| \leq C_1 \int_{\overline{r}}^r |g_1(t) - g_2(t)|\,dt 
\end{equation}
On the other hand, since the function \( \phi_p \) is of class \( C^1 \) and particularly Lipschitz in a neighborhood of \( h_0 \), we have:

\[
|g_1(t) - g_2(t)| \leq \int_{\overline{r}}^t \left(\frac{s}{t}\right)^{\theta-1} |\phi_p(h_1) - \phi_p(h_2)|\,ds \leq C \int_{\overline{r}}^t \left(\frac{s}{t}\right)^{\theta-1} |h_1(s) - h_2(s)|\,ds 
\]
\[ 
\leq C \left(\sup_{[\overline{r}, \overline{r}+\delta]} |h_1 - h_2|\right) \frac{t^{\theta}-\overline{r}^{\theta}}{\theta t^{\theta-1}} \leq C \delta \sup_{[\overline{r}, \overline{r}+\delta]} |h_1 - h_2| 
\]

Substituting this into the previous inequality, we get:
\[
|h_1 - h_2| \leq C_2 \delta \sup_{[\overline{r}, \overline{r}+\delta]} |h_1 - h_2| 
\]
where \( C_2 = C_1 C \).
Since we can assume \( \delta \) to be sufficiently small, for example \( \delta < \frac{1}{C_2} \), we have arrived at a contradiction. Indeed, we would have
\[\forall r \in [\overline{r}, \overline{r}+\delta],\quad |h_1 - h_2| < \sup_{[\overline{r}, \overline{r}+\delta]} |h_1 - h_2| .
\]

\textit{Subcase A.2) \( h'_0 = 0 = h_0 \).}\\
In this case, it is easy to show that the only solution is the trivial one. We have
\[
h(r) = -\int_{\overline{r}}^r \phi_{p'} \left( \int_{\overline{r}}^t \left(\frac{s}{t}\right)^{\theta-1}\phi_p(h)(s) \, ds \right) dt.
\]
If we assume, for contradiction, that \( h \neq 0 \) in \([\overline{r}, \overline{r}+\delta]\), then
\[ 
\begin{split}
|h(r)| &
\leq \int_{\overline{r}}^r \left(\int_{\overline{r}}^t \left(\frac{s}{t}\right)^{\theta-1} |h|^{p-1}\,ds\right)^{p'-1} dt 
\leq \left(\sup_{[\overline{r}, \overline{r}+\delta]} |h|\right)^{(p-1)(p'-1)} \int_{\overline{r}}^r \left(\frac{t^{\theta}-\overline{r}^{\theta}}{\theta t^{\theta-1}}\right) dt \\
& 
\leq \left(\sup_{[\overline{r}, \overline{r}+\delta]} |h|\right)^{(p-1)(p'-1)} \delta^{p'-1} \int_{\overline{r}}^r 1\,dt 
\leq \delta^{p'} \left(\sup_{[\overline{r}, \overline{r}+\delta]} |h|\right).
\end{split}
\]
Just as before, if we consider \( \delta \) to be sufficiently small, \(\delta<1\), this leads to the contradiction
\[ 
|h| < \sup_{[\overline{r}, \overline{r}+\delta]} |h| \quad  \forall r \in [\overline{r}, \overline{r}+\delta]. 
\]

\textit{Subcase A.3) \( h'_0 = 0 \) and \( h_0 \neq 0 \).}\\
In this case, the functions \( g_i(t) \), \( i = 1, 2 \), tend to 0 as \( t \to \overline{r} \). Therefore, we cannot use the local Lipschitz property of \( \phi_{p'} \) as we did in subcase A.1. Instead, a simple calculation shows that
\[
\lim_{t \to \overline{r}} \frac{1}{t-\overline{r}} \int_{\overline{r}}^t \left(\frac{s}{t}\right)^{\theta-1} \phi_p(h_i)(s) \, ds = \phi_p(h_0) \neq 0 .
\]

We can use the local Lipschitz property of \( \phi_{p'} \) in a neighborhood of \( \phi_p(h_0) \).

If \( h_1 \) and \( h_2 \) are two distinct solutions, we can write:
\[ 
|h_1 - h_2| \leq \int_{\overline{r}}^r \left| \phi_{p'} \left( \frac{t-\overline{r}}{t-\overline{r}}\int_{\overline{r}}^t \left(\frac{s}{t}\right)^{\theta-1} \phi_p(h_1) \right) - \phi_{p'} \left( \frac{t-\overline{r}}{t-\overline{r}} \int_{\overline{r}}^t \left(\frac{s}{t}\right)^{\theta-1} \phi_p(h_2) \right) \right| dt 
\]
\[ 
=\int_{\overline{r}}^r |t-\overline{r}|^{\frac{1}{p-1}} \left| \phi_{p'} \left( \frac{1}{t-\overline{r}}\int_{\overline{r}}^t \left(\frac{s}{t}\right)^{\theta-1} \phi_p(h_1) \right) - \phi_{p'} \left( \frac{1}{t-\overline{r}}\int_{\overline{r}}^t \left(\frac{s}{t}\right)^{\theta-1} \phi_p(h_2) \right) \right| dt 
\]
\[ 
\leq C_1 \int_{\overline{r}}^r |t-\overline{r}|^{\frac{1}{p-1}} \frac{1}{t-\overline{r}}\int_{\overline{r}}^t \left(\frac{s}{t}\right)^{\theta-1} |\phi_p(h_1)-\phi_p(h_2)| dsdt 
\]
\[ 
\leq C_1 C \int_{\overline{r}}^r (t-\overline{r})^{\frac{1}{p-1}-1}\int_{\overline{r}}^t \left(\frac{s}{t}\right)^{\theta-1} |h_1-h_2|dsdt 
\]
\[ 
\leq C_1 C \sup_{[\overline{r},\overline{r}+\delta]}|h_1-h_2| \int_{\overline{r}}^r (t-\overline{r})^{\frac{1}{p-1}} dt 
\]
\[ 
\leq C_1 C \delta^{p'} \sup_{[\overline{r},\overline{r}+\delta]}|h_1-h_2| 
\]
This inequality leads to a contradiction because if \( h_1 \neq h_2 \), unless we choose \( \delta \) sufficiently small such that \( C_1 C \delta^{p'} < 1 \), we get:
\[ 
|h_1-h_2| < \sup_{[\overline{r},\overline{r}+\delta]} |h_1-h_2|. 
\]

\textit{Case B) \( \overline{r} = 0 \), \( h'_0 = 0 \).}\\

\textit{Subcase B.1) \( h_0 = 0 \).}\\
This case is treated similarly to case A.2.\\

\textit{Subcase B.2) \( h_0 \neq 0 \).}\\
Let's implement an argument similar to case A.3.\\
Let \( h_1 \) and \( h_2 \) be as usual, two distinct solutions. We observe that for \( i = 1, 2 \)
\[ 
\left| \frac{1}{t} \int_0^t \left(\frac{s}{t}\right)^{\theta-1} \phi_p(h_i)ds - \frac{\phi_p(h_0)}{\theta } \right| \leq \frac{1}{t^{\theta}} \left| \int_0^t s^{\theta-1}|\phi_p(h_i)-\phi_p(h_0)|ds \right| \leq C \sup_{[0,t]} |h_i-h_0|. 
\]

Therefore, we have:
\[ 
\lim_{t \to 0} \frac{1}{t} \int_0^t \left(\frac{s}{t}\right)^{\theta-1}\phi_p(h_i)ds = \frac{\phi_p(h_0)}{\theta} \neq 0. 
\]
We can use the local Lipschitz property of \( \phi_{p'} \) in a neighborhood of \( \frac{\phi_p(h_0)}{\theta} \). We have
\[ 
|h_1 - h_2| \leq \int_0^r \left| \phi_{p'} \left( \frac{t}{t} \int_0^t \left(\frac{s}{t}\right)^{\theta-1} \phi_p(h_1) \right) - \phi_{p'} \left( \frac{t}{t} \int_0^t \left(\frac{s}{t}\right)^{\theta-1} \phi_p(h_2) \right) \,ds\right| dt 
\]
\[ 
\leq C_1 \int_0^r t^{p'-1} \left| \frac{1}{t} \int_0^t \left(\frac{s}{t}\right)^{\theta-1} (\phi_p(h_1)-\phi_p(h_2))ds \right| dt 
\]
\[ 
\leq C_1 C \int_0^r t^{p'-2} \int_0^t \left(\frac{s}{t}\right)^{\theta-1} |h_1-h_2| dsdt 
\]
\[ 
\leq C_1 C \sup_{[0,\delta]} |h_1-h_2| \int_0^r t^{p'-1} dt \]
\[ 
= C_2 \delta^{p'} \sup_{[0,\delta]} |h_1-h_2| 
\]

As usual, assuming \( \delta \) to be sufficiently small such that \( C_2 \delta^{p'} < 1 \), we arrive at a contradiction.\\\\

\textit{Step 4) Global Uniqueness.}\\
Let's assume that \( h_1 \) and \( h_2 \) are two solutions of the problem
\[ 
\begin{cases} (r^{\theta-1}\phi_p(h'))' + r^{\theta-1}\phi_p(h) = 0, & r \geq 0 \\
h(r_0) = h_0, \,\,\, h'(r_0) = h'_0 & \text{(\( h'_0 = 0 \) if \( r_0 = 0 \)).}
\end{cases} 
\]

We claim that \( h_1 \equiv h_2 \) in \([r_0, +\infty)\). Suppose, for the sake of contradiction, that this is not the case and let
\[ 
\tilde{r} := \inf\{r \geq r_0 : h_1(r) \neq h_2(r)\} 
\]

Then we have \( h_1(\tilde{r}) = h_2(\tilde{r}) \) and \( h'_1(\tilde{r}) = h'_2(\tilde{r}) \).

If \( \tilde{r} = r_0 \), then by our assumption, \( h_1(r_0) = h_2(r_0) \) and \( h'_1(r_0) = h'_2(r_0) \).

Now, consider the case when \( \tilde{r} > r_0 \). Due to the continuity of \( h_1 \) and \( h_2 \), if \( h_1(\tilde{r}) \neq h_2(\tilde{r}) \), it would imply that \( h_1(r) \neq h_2(r) \) in a neighborhood of \( \tilde{r} \). However, this contradicts the definition of \( \tilde{r} \).

Furthermore, for \( i = 1,2 \), we have:
\[ 
h'_i(\tilde{r}) = \phi_{p'}\left( \left(\frac{r_0}{\tilde{r}}\right)^{\theta-1}\phi_p(h_0) - \int_{r_0}^{\tilde{r}} \left(\frac{s}{\tilde{r}}\right)^{\theta-1}\phi_p(h_i)ds \right) 
\]

Since \( h_1 \equiv h_2 \) in \([r_0,\tilde{r}]\), we also conclude that \( h'_1(\tilde{r}) = h'_2(\tilde{r}) \).\\

By the local uniqueness from step 3, it follows that \( h_1 \equiv h_2 \) in \([\tilde{r},\tilde{r}+\delta]\), which contradicts the definition of \( \tilde{r} \). Therefore, the set over which we calculate the infimum is empty, showing that \( h_1 \equiv h_2 \) for \( r \geq r_0 \).

\end{proof}

\begin{proof}[Proof of Proposition \ref{1: PROP: Esistenza Unicità e Oscillazioni}]
The existence of the function \( \varphi \) follows from Theorem \ref{1: THM: Esistenza e Unicità}. If there exist a point \( z_0 \) such that \( \varphi(z_0) = 0 = \varphi'(z_0) \), then by uniqueness, \( \varphi \equiv 0 \) in \( [z_0, +\infty) \). Using a similar argument as in subcase A.2 and step 4, we would then have \( \varphi \equiv 0 \) for every \( r > 0 \), which is absurd. Therefore, every zero of the function \( \varphi \) is simple. 
We are left to prove that \( \varphi \) is oscillatory.\\

Assume for contradiction that \( \varphi \) is not oscillatory: \(\exists\, r_0 > 0 \) such that \( \varphi(r) \neq 0 \)  \(\forall r \geq r_0 \). 
For \( r \geq r_0 \), the function \( \dyle h(r) = r^{\theta-1}\phi_p\left(\frac{\varphi'(r)}{\varphi(r)}\right) \) is well-defined.
Using the fact that
\[ 
\frac{(r^{\theta-1}\phi_p(\varphi'))'}{\phi_p(\varphi)} = h' + \frac{(p-1)r^{\theta-1}\phi_p(\varphi')|\varphi|^{p-2}\varphi'}{\phi_p(\varphi)^2} = h' + \frac{(p-1)}{r^{(\theta-1)(p'-1)}} |h|^{p'} 
\]
and using the equation satisfied by \( \varphi \), we find that \( h \) is a solution of
\[
h' + \frac{(p-1)}{r^{(\theta-1)(p'-1)}}|h|^{p'} + r^{\theta-1} = 0 \quad \text{in} \quad [r_0,+\infty). 
\]
Let's prove that this leads to a contradiction.
Integrating from \( r_0 \) to \( t > r_0 \), we find,

\[ 
h(t) - h(r_0) + (p-1)\int_{r_0}^{t} \frac{|h|^{p'}}{s^{(\theta-1)(p'-1)}}\,ds + \frac{t^{\theta}}{\theta} - \frac{r_0^{\theta}}{\theta} = 0, \]

This yields the following system of equations:
\[
\begin{cases}
w(t) = \int_{r_0}^{t} \frac{|h|^{p'}}{s^{(\theta-1)(p'-1)}}\,ds \\
h(t) + (p-1)w(t) = -\frac{t^{\theta}}{\theta} + \frac{r_0^{\theta}}{\theta} + h(r_0).
\end{cases}
\]

\begin{enumerate}
\item \( w \geq 0 \). In particular, for \( t \) sufficiently large, \( h(t) \leq- ct^{\theta} \), so \( |h(t)| = -h(t) \geq ct^{\theta} \).
\item \(\dyle w(t) = \int_{r_0}^{t} \frac{|h|^{p'}}{s^{(\theta-1)(p'-1)}}\,ds \geq c_1 \int_{r_0}^{t} s^{\theta + p' - 1}\,ds = c_2 t^{\theta + p'} \) for \( t \) sufficiently large.
\item For \( t \) sufficiently large, \( h(t) + (p-1)w(t) \leq 0 \implies (p-1)w(t) \leq |h(t)| \) for \( t \gg r_0 \).
\item \(\dyle w'(t) = \frac{|h|^{p'}}{t^{(\theta-1)(p'-1)}} \).
\item \(\dyle (p-1)^{p'}w^{p'} \leq |h|^{p'} \frac{t^{(\theta-1)(p'-1)}}{t^{(\theta-1)(p'-1)}} = t^{(\theta-1)(p'-1)}w'(t) \) for \( t \gg r_0 \).
\end{enumerate}
%
%
%
%

Integrating from \( t \) to \( s > t \) this last inequality, we distinguish two cases:\\

\textit{Case A) \( (\theta-1)(p'-1) \neq 1 \), i.e., \( \theta \neq p \).}\\
In this case, we obtain 
\[ 
\int_t^s \frac{w'(\tau)}{w^{p'}(\tau)} \,d\tau \geq (p-1)^{p'}\int_t^s \tau^{\theta+p'-\theta p'-1} d\tau= \frac{(p-1)^{p'}}{(\theta+p'-\theta p')} 
(s^{\theta+p'-\theta p'}-t^{ \theta+p'- \theta p'})
\]
Therefore, 
\[ 
\frac{1}{w(t)^{p'-1}} - \frac{1}{w(s)^{p'-1}} \geq \frac{(p-1)^{p'+1}}{(\theta+p'- \theta p')} (s^{\theta+p'-\theta p'}-t^{\theta+p'-\theta p'}).
\]

\vspace{0.5cm}
\textit{Sub-case A.1) \( \theta + p' - \theta p' > 0 \) i.e. \( \theta < p \).}\\
We obtain a contradiction because from point 2., for \( t \gg r_0 \):
\[ 
\frac{c}{t^{(p'-1)(\theta+p')}} \geq \frac{1}{{w(t)}^{p'-1}} \geq \tilde{c} \left(s^{\theta+p'-\theta p'} - t^{\theta+p'-\theta p'}\right) + \frac{1}{{w(s)}^{p'-1}} 
\]
The left-hand side does not depend on \( s \) (which is arbitrary as long as \( s > t \)), while the right-hand side does. Letting \( s \to+ \infty \) leads to a contradiction. 
\\

\textit{Sub-case A.2) \( \theta + p' - \theta p' < 0 \) i.e. \( \theta > p \).}\\
In this case, we have:
\[ 
\frac{1}{{w(t)}^{p'-1}} - \frac{1}{{w(s)}^{p'-1}} \geq \frac{(p-1)^{p'+1}}{{\theta p' - \theta - p'}} \left(\frac{1}{{t^{\theta p' - \theta - p'}}} - \frac{1}{{s^{\theta p' - \theta - p'}}}\right). 
\]
Letting \( s \to +\infty \) and recalling that \( \frac{1}{w(s)^{p'-1}} \to 0 \), we obtain, for \( t \gg r_0 \):
\[ 
\frac{c}{{t}^{(p'-1)(\theta+p')}} \geq \frac{1}{{w(t)}^{p'-1}} \geq \tilde{c} \frac{1}{{t}^{\theta p' - \theta - p'}}
\]

This chain of inequalities is incompatible because \((p'-1)(\theta+p') = \theta p' - \theta - p' + (p')^2 > \theta p' - \theta - p'\). Once again, we have reached a contradiction.\\

\textit{Case B) \( (\theta-1)(p'-1) = 1 \) i.e. \( \theta = p \).}\\
Going back to point 5. and integrating from \( t \) to \( s > t \), we obtain:
\[ \frac{1}{{w(t)}^{p'-1}} - \frac{1}{{w(s)}^{p'-1}} \geq (p-1)^{p'+1} (\ln{s} - \ln{t}) \]
As in the previous cases, letting \( s \to +\infty \) leads to a contradiction.\\
\end{proof}

\begin{proof}[Proof of Proposition \ref{1: THM: Teorema Fondamentale}]
It is easy to verify that the pair \( \lambda_k, h_k \) is a solution. Conversely, suppose that for some \( \lambda \in \mathbb{R} \), we have a solution \( h(r) \) of \textbf{(P1)}. By multiplying the equation by \( h \) and integrating from 0 to \( R \), we can readily verify that \( \lambda \) must be positive. Thus, the function
\[ k(r) = h\left(\frac{r}{\lambda^{1/p}}\right) \text{ in } \left[0, \lambda^{1/p} R\right] \]
is well-defined, and it is easy to observe that \( k(r) \) solves:

%
%
%
%
\[
\begin{cases} 
(r^{\theta-1}\phi_p(k'))' + r^{\theta-1}\phi_p(k) = 0 \quad \text{in } [0, \lambda^{1/p} R] \\
k'(0) = h'(0) = 0,\,\,\, 
k(\lambda^{1/p} R) = h(R) = 0 
\end{cases}
\]
\vspace{0.3cm}

We also note that \( k(0) = h(0) \). If \( h(0) = 0 \), then \( k(r) \) would be a local solution of the equation
\[
(r^{\theta-1}\phi_p(k'))' + r^{\theta-1}\phi_p(k) = 0 
\]

with zero initial data, \( k(0) = k'(0) = 0 \). By Theorem \ref{1: THM: Esistenza e Unicità} on existence and uniqueness, it would imply \( k \equiv 0 \), which means \( h \equiv 0 \) in \([0,R]\). 
Unless we exclude this trivial solution, we can assume \( h(0) \neq 0 \).\\
The function \( \dyle \tilde{k}(r) = \frac{1}{h(0)} k(r) \) solves
\[
\begin{cases} 
(r^{\theta-1}\phi_p(\tilde{k}'))' + r^{\theta-1}\phi_p(\tilde{k}) = 0 \quad \text{in } [0, \lambda^{1/p} R] \\
\tilde{k}(0) = 1, \quad \tilde{k}'(0) = 0.
\end{cases}
\]
By uniqueness, we have \( \tilde{k} \equiv \varphi(r) \), which implies \(\dyle \frac{1}{h(0)} h\left(\frac{r}{\lambda^{1/p}}\right) = \varphi(r) \). From this, we deduce \( h(r) = h(0) \varphi\left(\lambda^{1/p} r\right) \) over \([0,R]\). Since \( h(R) = 0 \), it follows that \( \varphi\left(\lambda^{1/p} R\right) = 0 \), implying \( \lambda^{1/p} R = \nu_k(p,\theta) \) for some \( k \geq 1 \). Consequently, we can conclude:
\[
\begin{cases} 
\lambda = \lambda_k = \left(\frac{\nu_k(p,\theta)}{R}\right)^p \\
h = h_k(r) = h(0) \varphi\left(\frac{\nu_k(p,\theta)}{R} r\right).
\end{cases}
\]
\end{proof}

\section{Appendix}\label{Appendice B}
In this appendix, we provide a proof of Proposition \ref{STIME: PROP: STIMA DEL PRIMO TIPO} and Proposition \ref{STIME: PROP: STIMA DEL SECONDO TIPO}.\\\\
In order to approach the proof of Proposition \ref{STIME: PROP: STIMA DEL PRIMO TIPO}, we need some preliminary lemmas.

\begin{lemma}\label{STIME: LEMMA: 1: STIMA DEL PRIMO TIPO}
Let \(p\geq 3\), \( a\in\left[0,\frac{1}{2}\right] \)and let \(k(a)\coloneqq (1-a)^{p-1}-(p-1)(1-a)^{p-2}\), then 
\[
\max_{a\in\left[0,\frac{1}{2}\right]} k(a)
=
k\left(\frac{1}{2}\right)
=
\frac{3-2p}{2^{p-1}}.
\]
\end{lemma}
\begin{proof}
Studying the derivative, it is easily seen that the function is monotonically increasing in the interval \([0, 1/2]\).
\end{proof}

Let's define
\[
h(a)\coloneqq p(p-1)\int_0^1s\abs{a-s}^{p-2}\,ds\quad a\in[0,1].
\]

\begin{lemma}\label{STIME: LEMMA: 2: STIMA DEL PRIMO TIPO}
For every \(p\geq 2\),
\(
h(a)=a^p+(1-a)^{p-1}(a+p-1)\quad a\in[0,1].
\)
\end{lemma}
\begin{proof}
Since \(a \in [0,1]\), we can write
\[
\begin{split}
&h(a)
=
p(p-1)\int_0^a s(a-s)^{p-2}\,ds
+p(p-1)\int_a^1 s(s-a)^{p-2}\,ds\\
&
=
p(p-1)\int_0^a (a-t)t^{p-2}\,dt
+p(p-1)\int_0^{1-a} (a+t)t^{p-2}\,dt\\
&
=
p(p-1)\left[\frac{at^{p-1}}{p-1}-\frac{t^p}{p} \right]_0^a
+p(p-1)\left[\frac{at^{p-1}}{p-1}+\frac{t^p}{p} \right]_0^{1-a}\\
&
=
a^p+(1-a)^{p-1}(a+p-1)
\end{split}
\]
\end{proof}

\begin{lemma}\label{STIME: LEMMA: 3: STIMA DEL PRIMO TIPO}
Let \(p \geq 3\), then the minimum of the function \(h(a)\) is attained in the interval \([1/2,1]\):
\[
\min_{[0,1]}h(a)=\min_{[1/2,1]}h(a)
\]
\end{lemma}
\begin{proof} If \(a\in (0,1/2)\), then
\( h'(a)=pa^{p-1}+p\,k(a)\leq \frac{2p}{2^{p-1}}(2-p)
<0. \)
In the interval \( \left(0,\frac{1}{2}\right)\), the function \(h(a)\) is strictly decreasing, thus it must attain the minimum in \( \left[\frac{1}{2}, 1\right]\).
\end{proof}

\begin{proof}[Proof of Proposition \ref{STIME: PROP: STIMA DEL PRIMO TIPO}]
%
If \( y=0 \) or \( x=y \), there is nothing to prove. Without loss of generality, we can assume \(x\neq y\) and \(y\neq 0\). Let \(z=x/y\), the goal is to prove that the function
\[
g(z)\coloneqq
p(p-1)\frac{\int_0^1 s|s+(1-s)z|^{p-2}\,ds}{|z-1|^{p-2}}
=
p(p-1)\int_0^1 s\left\lvert \frac{z}{z-1}-s\right\rvert^{p-2}\,ds\quad z\neq 1,
\]
is bounded below by a constant \(\lambda_p\) as in the thesis.\\

\textit{First Step:}
\(\lambda_p>0\).\\
Since \(g(z)>0\) in \(\mathbb{R}\setminus\{1\}\), the only potential problems could arise at \(z=1\) and at infinity.\\
If \(|z|\to+\infty\), then
\[
\lim_{|z|\to+\infty}g(z)
=
p(p-1)\int_0^1 s(1-s)^{p-2}\,ds
=
1>0.
\]
In particular, from here, we see that \(\inf_{z\neq 1}g(z)\leq 1\).\\
If \(z\to 1\),
\[
\lim_{z\to 1} g(z)=+\infty.
\]
Therefore, in a neighborhood of \(z=1\), \(g(z)\) is strictly positive.
We can conclude that \( 1\geq \lambda_p=\inf_{z\neq 1}g(z)>0.\)\\

\textit{Second Step:} \(\lambda_p \in \left[\frac{1}{2^p},\frac{p}{2^{p-1}}\right]\).\\
If \(p=2\), it is immediately verified that \(\lambda_p=1 \in[1/4,1]\). We can assume \(p\geq 3\).
 For obvious geometric reasons, the following chain of equalities holds:
\[
\lambda_p
=
\inf_{z\neq 1} p(p-1)\int_{0}^1 s\left\lvert \frac{z}{z-1}-s\right\rvert^{p-2}\,ds
=
\inf_{a\in\mathbb{R}} p(p-1)\int_0^1 s|a-s|^{p-2}\,ds
=
\min_{a\in[0,1]} p(p-1)\int_0^1 s|a-s|^{p-2}\,ds.
\]
From Lemmas \ref{STIME: LEMMA: 2: STIMA DEL PRIMO TIPO} and \ref{STIME: LEMMA: 3: STIMA DEL PRIMO TIPO}, we thus have \( \lambda_p=\min_{a\in[0,1]}h(a)=\min_{a\in\left[\frac{1}{2},1\right]} h(a).\)
On one hand, we can conclude that 
\[
\lambda_p\leq h\left(\frac{1}{2}\right)
=
\frac{1}{2^p}+\frac{1}{2^{p-1}}\left[p-\frac{1}{2}\right]
=\frac{p}{2^{p-1}}.
\]
On the other hand,
\[
\lambda_p
=
\min_{a\in\left[\frac{1}{2},1\right]} h(a)
=
\min_{a\in\left[\frac{1}{2},1\right]} \left( a^p 
+(1-a)^{p-1}(a+p-1)\right)
\geq \frac{1}{2^p}.
\]
\end{proof}

\vspace{0.6cm}
In order to prove Proposition \ref{STIME: PROP: STIMA DEL SECONDO TIPO}, we need a preliminary lemma.

\begin{lemma}\label{STIME: LEMMA: 1: STIMA DEL SECONDO TIPO}
For every \(p\geq 2\), we have
\begin{itemize}
    \item \(\dyle x^p-px+(p-1)\geq (p-1)(1-x)^2\) for \(x\geq 0\)
    \item \(\dyle x^p+px+(p-1)\geq \frac{p}{2}(1+x)^2\) for \(x\geq 0\)
\end{itemize}
\end{lemma}

\begin{proof}
The second inequality follows directly from the first one and some algebraic manipulation.
Let's prove the first inequality. Let \(g(x)=x^p-px+(p-1)-(p-1)(1-x)^2= x\left(x^{p-1}-(p-1)x+(p-2)\right)\)
Then \(g(x)\geq 0\) for \(x\geq 0 \iff\) \(h(x)=x^{p-1}-(p-1)x+(p-2)\geq 0\) for \(x\geq 0\). But
\[
h'(x)=(p-1)x^{p-2}-(p-1)=(p-1)(x^{p-2}-1).
\]
So the function \(h\) has a minimum point at \(x=1\) where \(h(1)=0\).
\end{proof}

\begin{proof}[Proof of Proposition \ref{STIME: PROP: STIMA DEL SECONDO TIPO}]
Without loss of generality, we can assume \(y\neq x\) and \(y\neq 0\), since in these cases the inequality is immediate.
Dividing by \(|y|^{p-2}\), we need to demonstrate that
\[
p(p-1)\int_0^1 s\left| s+(1-s)\frac{x}{y}\right|^{p-2}\,ds
\geq\frac{p}{2}.
\]

Renaming \(\dyle \frac{x}{y}\) as \(x\), we define the function \(f(x)\) as follows

\[ f(x)=p(p-1)\int_0^1 s|s+(1-s)x|^{p-2}\,ds \quad x\in\mathbb{R}\setminus \{1\}. \]

With the change of variable \(t=x+(1-x)s\), we get:

\[ f(x)=\frac{p(p-1)}{(1-x)^2}\int_1^x (x-t)|t|^{p-2}\,dt. \]

\textit{First Case: \(x\geq 0\).}\\
\[
\begin{split}
f(x)=\frac{p(p-1)}{(1-x)^2}\int_1^x (x-t)t^{p-2}\,dt
=
\frac{p(p-1)}{(1-x)^2}\left[\frac{xt^{p-1}}{p-1}-\frac{t^p}{p}\right]^x_1
=
\frac{x^p-px+(p-1)}{(1-x)^2}
\geq
(p-1)
\geq\frac{p}{2}
\end{split}
\]
thanks to Lemma \ref{STIME: LEMMA: 1: STIMA DEL SECONDO TIPO}.

\textit{Second Case: \(x<0\).}\\
\[
\begin{split}
&f(x)=\frac{p(p-1)}{(1-x)^2}\int_x^1(t-x)|t|^{p-2}\,dt
=
\frac{p(p-1)}{(1-x)^2}\left[\int_x^0(t-x)|t|^{p-2}\,dt+\int_0^1(t-x)t^{p-2}\,dt\right]\\[10pt]
&
=
\frac{p(p-1)}{(1-x)^2}\left\{-\int_0^{-x}(t+x)t^{p-2}\,dt+\left[\frac{t^p}{p}-\frac{xt^{p-1}}{p-1}\right]^1_0\right\}
\quad\text{let }y=-x,
\\[10pt]
&
=
\frac{p(p-1)}{(1+y)^2}\left\{-\int_0^y(t^{p-1}-yt^{p-2})\,dt+\frac{1}{p}+\frac{y}{p-1}\right\}
=
\frac{y^p+py+(p-1)}{(1+y)^2}
\geq
\frac{p}{2}
\end{split}
\]
thanks to Lemma \ref{STIME: LEMMA: 1: STIMA DEL SECONDO TIPO}.\\

We have thus shown that \(f(x)\geq \frac{p}{2}=f(-1).\)
\end{proof}

\section*{Acknowledgements}
The author is partially supported by Progetti per Avvio alla Ricerca di Tipo 1 – Sapienza University of Rome, prot. no. AR123188AF021F31 and prot. no. AR1241906F35D8DC.

\nocite{*}
\bibliographystyle{plain} 
\bibliography{Bibliografia}
\end{document}